\documentclass[a4paper]{scrartcl}

\usepackage[utf8]{inputenc}
\usepackage[T1]{fontenc} % For proper typesetting of accented characters.

\usepackage{amsmath, amsthm, amssymb}
\usepackage{mathrsfs} % For \mathscr, a calligraphic font
\usepackage{enumitem}
\usepackage{moreenum} % For enumeration with greek letters
\newlist{hypothenum}{enumerate}{3}
\setlist[hypothenum,1]{label=(\roman*)}
\usepackage{url} % For typesetting URLs (duh!)
\usepackage{needspace} % Very convenient way to prevent line breaks
\usepackage{framed}
\usepackage{tikz-cd} % For commutative diagrams
\usepackage[all,cmtip]{xy} % Useful for the section "complicated bijections"
\usepackage{graphics} % Necessary to include Metapost graphics
\usepackage{colonequals} % For typesetting :=

\theoremstyle{plain}
\newtheorem*{theorem*}{Theorem}
\newtheorem{theorem}{Theorem}[subsection]
\newtheorem*{proposition*}{Proposition}

\newtheorem*{lemma*}{Lemma}
\newtheorem*{lemmabis*}{Lemma bis}
\newtheorem{lemma}[theorem]{Lemma}
\newtheorem*{corollary*}{Corollary}

\newtheorem{corollary_global}{Corollary}
\newtheorem*{claim*}{Claim}

\theoremstyle{definition}
\newtheorem*{definition*}{Definition}

\newtheorem*{example*}{Example}

\theoremstyle{remark}
\newtheorem*{remark*}{Remark}
\newtheorem*{remarks*}{Remarks}

\numberwithin{equation}{section}

\newcommand{\eps}{\varepsilon}
\newcommand{\setsuch}[2]{\left\{ #1 \; \middle| \; #2 \right\}} % The notation { x in R | x has some property }
\newcommand{\restr}[2]{{\left. #1 \right|}_{#2}} % The notation f|X for the restriction of f to X

\newcommand{\transpose}[1]{{}^t#1}
\newcommand{\bigo}{\mathcal{O}}
\DeclareMathOperator{\rank}{rank}

\DeclareMathOperator{\Image}{Im}
\DeclareMathOperator{\Ker}{Ker}
\DeclareMathOperator{\Ad}{Ad}

\DeclareMathOperator{\End}{End}
\DeclareMathOperator{\GL}{GL}
\DeclareMathOperator{\SL}{SL}
\DeclareMathOperator{\PSL}{PSL}

\DeclareMathOperator{\SU}{SU}

\newcommand{\fundef}[5]{
\entrymodifiers={+!!<0pt,\fontdimen22\textfont2>}
\xymatrix@R=3pt{\llap{$#1$\;\;} {#2} \ar@{->}[r] & {#3} \\ {#4} \ar@{|->}[r] & {#5}}
} %Formats function definitions nicely. For more complicated things (bijections with two arrows etc.), redo it by hand. To explain the "\entrymodifiers" line: see https://www.math.lsu.edu/~aperlis/publications/axisalignment/perlis_axisalignment_24Jun2003.pdf

\newcommand{\ie}{i.e.\ }

\begin{document}

\title{Asymptotic properties of linear groups}
\author{Yves Benoist\footnote{Translated from French by Ilia Smilga. Original title: Propriétés asymptotiques des groupes linéaires, {\em Geom. and funct. anal.}, 7:1--47, 1997.}}
\date{\vspace{-5ex}} % Pour enlever la ligne de la date
\maketitle

\begin{abstract}
Let $G$ be a reductive linear real Lie group and $\Gamma$ be a Zariski dense subgroup. We study asymptotic properties of $\Gamma$ through the set of logarithms of the radial components of the elements of $\Gamma$: we prove that the asymptotic cone of this set is a convex cone with non empty interior and is stable by the Cartan involution. Conversely, any closed convex cone of the positive Weyl chamber whose interior is non empty and which is stable by the opposition involution can be obtained this way.

We relate this limit cone and the limit set of $\Gamma$ to the set of open semigroups of $G$ which meet $\Gamma$.

We also prove similar results over any local fields.
\end{abstract}

\section{Introduction}
\label{sec:1}

The goal of this paper is to study certain asymptotic properties of the subgroups~$\Gamma$ of the linear group~$\GL(V)$ of a finite-dimensional vector space~$V$ over the field $k = \mathbb{R}$ (and more generally over a local field~$k$) when $V$~is completely reducible, \ie a direct sum of irreducible invariant subspaces. In this case, the Zariski closure of~$\Gamma$ is reductive.

In other terms, we study asymptotic properties of Zariski-dense subgroups of the group~$G$ of $k$-points of a reductive $k$-group.

\makeatletter
\renewcommand\subsection{\@startsection{subsection}{2}{\z@}%
                                     {-3.25ex\@plus -1ex \@minus -.2ex}%
                                     {-1.5ex \@plus .2ex}%
                                     {\normalfont\large\bfseries}}
\makeatother %This (temporarily) suppresses the line break after subsection number

\subsection{}
\label{sec:1.1}
Let $G$~be a connected linear reductive real Lie group, $Z$~its center, $A_G$~a Cartan subspace of~$G$, $A^+$~a closed Weyl chamber of~$A_G$ and $K$~a maximal compact subgroup of~$G$ for which we have the Cartan decomposition: $G = K A^+ K$. Let us denote by $\mu: G \to A^+$ the Cartan projection: for $g$ in~$G$, $\mu(g)$ is the radial component of~$g$, \ie the unique element of $A^+ \cap K g K$. The logarithm map~$\log$ identifies $A_G$ with its Lie algebra~$\mathfrak{a}$. Let us denote by $\mathfrak{a}^+ := \log A^+$ the Weyl chamber of~$\mathfrak{a}$.

Let $\Gamma$ be a Zariski-dense subsemigroup of~$G$. We are interested in the asymptotic properties of~$\Gamma$. Since $\mu$ is a proper and continuous map, some of these properties can be read off the asymptotic properties of the set $\log(\mu(\Gamma))$ of the logarithms of the radial components of the elements of~$\Gamma$. Our goal is to describe the asymptotic cone to this set, \ie the cone of~$\mathfrak{a}$ formed by the limit directions of the sequences of elements in this set that go away to infinity. For this, let us introduce a few notations.

Let $\lambda: G \to A^+$ be the natural projecion that comes from the Jordan decomposition and $\imath : \mathfrak{a}^+ \to \mathfrak{a}^+$ be the opposition involution: for $g$ in~$G$, $\lambda(g)$~is the unique element of $A^+$ that is conjugate to the hyperbolic component $g_h$ of~$g$; for $X$ in~$\mathfrak{a}^+$, $\imath(X)$~is the unique element of~$A^+$ that is conjugate to~$-X$. Let $\ell_\Gamma$ denote the smallest closed cone of~$\mathfrak{a}^+$ that contains $\log(\lambda(\Gamma))$. We shall call it the limit cone of~$\Gamma$. When $\Gamma$~is a subgroup of~$G$, this cone is invariant by the opposition involution.

\subsection{}
\label{sec:1.2}
One of the main results of this paper is the following :
\begin{theorem*}
Let $G$ be a connected semisimple linear real Lie group.
\begin{enumerate}[label=\alph*)]
\item Let $\Gamma$ be a Zariski-dense subsemigroup of~$G$. Then:
\begin{enumerate}[label=\greek*)]
\item The asymptotic cone to $\log(\mu(\Gamma))$ is the limit cone $\ell_\Gamma$.
\item The limit cone $\ell_\Gamma$ is convex and has nonempty interior.
\end{enumerate}
\item Conversely, suppose $G$ is not compact. Let $\Omega$ be a closed convex cone with nonempty interior in~$\mathfrak{a}^+$.
\begin{enumerate}[label=\greek*)]
\item Then there exists a Zariski-dense discrete subsemigroup $\Gamma$ of~$G$ such that ${\ell_\Gamma = \Omega}$.
\item If additionally $\Omega$ is stable by the opposition involution, there exists a Zariski-dense discrete subgroup $\Gamma$ of~$G$ such that $\ell_\Gamma = \Omega$.
\end{enumerate}
\end{enumerate}
\end{theorem*}

\begin{example*}
For a reader with little familiarity with the theory of semisimple Lie groups, let us work through this theorem in the particular case where $G = \SL(n, \mathbb{R})$.

In this case, $\mathfrak{a}$~is the set of diagonal matrices with zero trace, that we identify with the hyperplane given by
\[\mathfrak{a} = \setsuch{x = (x_1, \ldots, x_n) \in \mathbb{R}^n}{x_1 + \cdots + x_n = 0};\]
$\mathfrak{a}^+$ is the convex cone given by
\[\mathfrak{a}^+ = \setsuch{x = (x_1, \ldots, x_n) \in \mathfrak{a}}{x_1 \geq \cdots \geq x_n};\]
and the opposition involution is the linear map given by
\[\imath(x_1, \ldots, x_n) = (-x_n, \ldots, -x_1).\]
For $g$ in~$G$, the vector $m_g := \log(\mu(g))$ is the vector of~$\mathfrak{a}^+$ whose coordinates are the logarithms of the eigenvalues of the symmetric matrix $(\transpose{g} g)^\frac{1}{2}$ sorted in nonincreasing order and the vector $\ell_g := \log(\lambda(g))$ is the vector of~$\mathfrak{a}^+$ whose coordinates are the logarithms of the moduli of the eigenvalues of~$g$ sorted in nonincreasing order. The cone~$\ell_\Gamma$ is simply the closure of the set of half-lines generated by nonzero vectors~$\ell_g$ for $g$ in~$\Gamma$.
\end{example*}

\subsection{}
\label{sec:1.3}
A second concept that reflects some of the asymptotic properties of~$\Gamma$ is the limit set~$\Lambda_\Gamma$: it is a closed subset of the full flag variety of~$G$, which is a classical object when $G$~has real rank one and which has been introduced by Y.~Guivarc'h for $G = \SL(n, \mathbb{R})$. In general, this object differs from the what we call~$\Lambda^-_\Gamma$, which is the corresponding  object for the semigroup~$\Gamma^-$ formed by the inverses of the elements of~$\Gamma$.

A third concept is the set~$\mathcal{L}_\Gamma$ of the limit directions of~$\Gamma$: this is the set of the hyperbolic elements~$g$ of~$G$ such that every open semigroup~$H$ of~$G$ that contains~$g$ intersects~$\Gamma$. Heuristically, this set~$\mathcal{L}_\Gamma$ describes the directions at infinity in~$G$ that sequences of elements of~$\Gamma$ may follow. We show how to express this set~$\mathcal{L}_\Gamma$ in terms of $\Lambda_\Gamma, \Lambda^-_\Gamma$ and~$L_\Gamma$ and vice-versa (Theorem~\ref{sec:6.4}). In particular, we prove that:
\emph{
\begin{itemize}
\item The set~$\mathcal{L}_\Gamma$ intersects the interior of a Weyl chamber~$f$ if and only if the ``starting point''~$y^-_f$ of~$f$ is in~$\Lambda^-_\Gamma$ and the ``end point''~$y^+_f$ is in~$\Lambda_\Gamma$.
\item In this case, the intersection $\mathcal{L}_\Gamma \cap f$ ``does not depend'' on the choice of the chamber: it can be naturally identified with the limit cone $L_\Gamma$.
\end{itemize}
}

\subsection{}
\label{sec:1.4}
In the following part of the text, we also prove an analog of these results for a reductive group over an arbitrary local field~$k$: in this case, the logarithm map is replaced by an injection of the positive Weyl chamber~$A^+$ into a cone $A^\times$ of an $\mathbb{R}$-vector space~$A^\bullet$ (the latter two may be identified respectively with~$\mathfrak{a}^+$ and~$\mathfrak{a}$ when $k = \mathbb{R}$) and we define once again maps $\mu: G \to A^+$, $\lambda: G \to A^\times$, $\imath: A^\times \to A^\times$ (cf. \ref{sec:2.3} and~\ref{sec:2.4}) and, for every subsemigroup $\Gamma$ of~$G$, a limit cone~$L_\Gamma$ in~$A^\times$ (which may be identified to~$\ell_\Gamma$ when $k = \mathbb{R}$). By construction this cone~$L_\Gamma$ is closed and with rational support (\ie $L_\Gamma$ and~$L_\Gamma \cap A^+$ generate the same vector subspace of~$A^\bullet$). Additionally, if $\Gamma$~is discrete, this cone is not central (\ie is not contained in the limit cone~$L_Z$ of the center of~$G$). The main difference compared to the case when $k = \mathbb{R}$ is that the cone~$L_\Gamma$ may have empty interior in~$A^\times$.

In this case, Theorem~\ref{sec:1.2} becomes:

\begin{theorem*}
Let $k$ be a non-Archimedean local field, $\mathbf{G}$ a connected reductive $k$-group and $G = \mathbf{G}_k$.
\begin{enumerate}[label=\alph*)]
\item Let $\Gamma$ be a Zariski-dense subsemigroup of~$G$. Then:
\begin{enumerate}[label=\greek*)]
\item The asymptotic cone to $\mu(\Gamma)$ is the limit cone $L_\Gamma$.
\item The limit cone $L_\Gamma$ is convex.
\end{enumerate}
\item Conversely, let $\Omega$ be a non-central closed convex cone with rational support in~$A^\times$.
\begin{enumerate}[label=\greek*)]
\item Then there exists a Zariski-dense discrete subsemigroup $\Gamma$ of~$G$ such that ${L_\Gamma = \Omega}$.
\item If additionally $\Omega$ has nonempty interior and is stable by the opposition involution, there exists a Zariski-dense discrete subgroup $\Gamma$ of~$G$ such that $L_\Gamma = \Omega$.
\end{enumerate}
\end{enumerate}
\end{theorem*}

It is likely that in this last statement, the assumption ``$\Omega$ has nonempty interior'' is not necessary (I verified it for $G = \SL(n, k)$ and $\Omega$ a half-line).

\subsection{}
\label{sec:1.5}
The open subsemigroups of~$G$ play an important role in this work. They serve as a source of examples (cf. \ref{sec:5.2} and~\ref{sec:5.3}) but also as a tool (cf.~\ref{sec:6.3}). Another important tool for the proofs is the notion of a $(\theta, \underline{\eps})$-Schottky subsemigroup or subgroup: it generalizes the $\eps$-Schottky subgroups that I introduced in~\cite{Be}. The latter will be another source of examples (cf.~\ref{sec:5.1}). The $(\theta, \underline{\eps})$-Schottky subsemigroups will also appear in the proof of the convexity of~$L_\Gamma$ (cf.~\ref{sec:4.4}). Let us cite \cite{Ti2}, \cite{Ma-So} and \cite{A-M-S} where similar notions are used.

A few words about the plan of the paper. Section~\ref{sec:2} mostly consists of reminders of~\cite{Be}. Section~\ref{sec:3} is dedicated to the limit set~$\Lambda_\Gamma$, section~\ref{sec:4} to the convexity of~$L_\Gamma$, section~\ref{sec:5} to the construction of subsemigroups and subgroups~$\Gamma$ of~$G$ whose cone~$L_\Gamma$ is prescribed, section~\ref{sec:6} to properties of the set $\mathcal{L}_\Gamma$. Finally, in section~\ref{sec:7}, we verify that, when $k = \mathbb{R}$, the cone~$L_\Gamma$ has nonempty interior.

\makeatletter
\renewcommand\subsection{\@startsection{subsection}{2}{\z@}%
                                     {-3.25ex\@plus -1ex \@minus -.2ex}%
                                     {1.5ex \@plus .2ex}%
                                     {\normalfont\large\bfseries}}
\makeatother %This restores ordinary (?) subsection style

\section{Preliminaries}
\label{sec:2}

We recall in this section some notations introduced in~\cite{Be}.

\subsection{Local fields}
\label{sec:2.1}

Let $k$~be a local field, \ie either $\mathbb{R}$ or $\mathbb{C}$ or a finite extension of~$\mathbb{Q}_p$ or of~$\mathbb{F}_p((T))$ for some prime integer~$p$. Let $|.|$~be a continuous absolute value on~$k$.

When $k = \mathbb{R}$ or~$\mathbb{C}$, we set $k^o := (0, \infty)$ and $k^+ := [1, \infty)$.

When $k$~is non-Archimedean, we call $\mathcal{O}$ the ring of integers of~$k$, $\mathcal{M}$~the maximal ideal of~$\mathcal{O}$ and we choose a uniformizer, \ie an element~$\pi$ of~$\mathcal{M}^{-1}$ which is not in~$\mathcal{O}$. We then set $k^o := \setsuch{\pi^n}{n \in \mathbb{Z}}$ and $k^+ := \setsuch{\pi^n}{n \geq 0}$.

Let $V$ be a finite-dimensional vector space over~$k$. To every basis $v_1, \ldots, v_n$ of~$V$, we associate norms on~$V$ and on~$\End(V)$ defined, for every $v = \sum_{1 \leq i \leq n} x_i v_i$ in~$V$ and for every $g$ in~$\End(V)$, by
\[\|v\| := \sup_{1 \leq i \leq n} |x_i|
\quad\text{and}\quad
\|g\| := \sup_{v \in V,\; \|v\| = 1} \|g \cdot v\|.\]

Of course two different bases of~$V$ give rise to equivalent norms.

We call $X := \mathbb{P}(V)$ the projective space of~$V$. We define a distance~$d$ on~$X$ by
\[d(x_1, x_2) := \inf \setsuch{\|v_1 - v_2\|}{v_i \in x_i \text{ and } \|v_i\| = 1\quad \forall i = 1, 2}.\]
If $X_1$ and~$X_2$ are two closed subsets of~$X$, we set
\[\delta(X_1, X_2) := \inf \setsuch{d(x_1, x_2)}{x_1 \in X_1,\; x_2 \in X_2} \text{ and}\]
\[d(X_1, X_2) := \sup \setsuch{\delta(x_i, X_{3-i})}{x_i \in X_i \text{ and } i = 1, 2}\]
the Hausdorff distance between $X_1$ and~$X_2$.

We denote by $\lambda_1(g) \geq \cdots \geq \lambda_n(g)$ the sequence of moduli of eigenvalues of~$g$ sorted in nonincreasing order and repeated according to multiplicity. Of course an eigenvalue of~$g$ is in general in some finite extension~$k'$ of~$k$. We implicitly endowed this extension with the unique absolute value that extends the absolute value of~$k$.

\subsection{Proximality}
\label{sec:2.2}

An element~$g$ of~$\End(V) \setminus 0$ is said to be \emph{proximal in $\mathbb{P}(V)$} or \emph{proximal} if it has a unique eigenvalue~$\alpha$ such that $|\alpha| = \lambda_1(g)$ and this eigenvalue has multiplicity one. This eigenvalue~$\alpha$ is then in~$k$. We call $x^+_g \in X$ the corresponding eigenline, $V^<_g$ the $g$-invariant hyperplane supplementary to~$x^+_g$ and $X^<_g := \mathbb{P}(V^<_g)$.

We fix $\eps > 0$ and we define
\[b^\eps_g := \setsuch{x \in X}{d(x, x^+_g) \leq \eps};\]
\[B^\eps_g := \setsuch{x \in X}{\delta(x, X^<_g) \geq \eps}.\]
We say that a proximal element~$g$ is $\eps$-proximal if $\delta(x^+_g, X^<_g) \geq 2\eps$, $g(B^\eps_g) \subset b^\eps_g$ and $\restr{g}{B^\eps_g}$ is $\eps$-Lipschitz. The following lemma is easy (cf. Corollary~6.3 in \cite{Be}).

\begin{lemma}
\label{lem_2.2.1}
For every $\eps > 0$, there exists a constant $c_\eps = c_\eps(V) \in (0, 1)$ such that, for every $\eps$-proximal linear transformation~$g$ of~$V$, we have
\[c_\eps \|g\| \leq \lambda_1(g) \leq \|g\|.\]
\end{lemma}

The following lemma is a variant of Proposition~6.4 in~\cite{Be}. It can be proved in the same fashion.

\begin{lemma}
\label{lem_2.2.2}
For all $\eps > 0$, there exist constants $C_\eps > 0$ with the following property. Take $g_1, \ldots, g_l$ to be any linear transformations of~$V$ that are respectively $\eps_1$-proximal, $\ldots$, $\eps_l$-proximal and that satisfy (with the convention $g_0 = g_l$)
\[\delta(x^+_{g_{j-1}}, X^<_{g_j}) \geq 6 \sup(\eps_{j-1}, \eps_j) \quad\text{for } j = 1, \ldots, l.\]
Then for any $n_1, \ldots, n_l \geq 1$, the product $g := g_l^{n_l} \cdots g_1^{n_1}$ is $\eps$-proximal for $\eps = 2 \sup(\eps_1, \eps_l)$.

Moreover, setting $\lambda_1 := \prod_{1 \leq j \leq l} \lambda_1(g_j)^{n_j}$ and $C := \prod_{1 \leq j \leq l} C_{\eps_j}$, we have
\[\lambda_1(g) \in [\lambda_1 C^{-1}, \lambda_1 C]
\quad\text{and}\quad
\|g\| \in [\lambda_1 C^{-1}, \lambda_1 C].\]
\end{lemma}

\begin{definition*}
Let $\underline{\eps} = (\eps_j)_{j \in J}$ be a finite or infinite family of cardinal $t \geq 2$ of positive real numbers. We say that a subsemigroup (resp.\ subgroup) $\Gamma$ of~$\GL(V)$ with generators $(\gamma_j)_{j \in J}$ is \emph{$\underline{\eps}$-Schottky on $\mathbb{P}(V_i)$} if it satisfies the following properties. (We set $E_\Gamma := \setsuch{\gamma_j}{j \in J}$, resp.\ $E_\Gamma := \setsuch{\gamma_j, \gamma^{-1}_j}{j \in J}$; and for every element $g$ of~$E_\Gamma$ whose index is~$j$, we set $\eps_g := \eps_j$.)
\begin{enumerate}[label=\roman*)]
\item For every $g$ in~$E_\Gamma$, $g$~is $\eps_g$-proximal.
\item For every $g, h$ in~$E_\Gamma$ (resp.\ $g, h$ in~$E_\Gamma$ such that $g \neq h^{-1}$), $\delta(x^+_g, X^<_h) \geq 6\sup(\eps_g, \eps_h)$.
\end{enumerate}
\end{definition*}

\begin{remark*}
When the tuple~$\underline{\eps}$ is constant and equal to some $\eps > 0$, we say that $\Gamma$ is \emph{$\eps$-Schottky on~$\mathbb{P}(V_i)$} (in \cite{Be}, such groups were called ``$\eps$-proximal'').
\end{remark*}

\subsection{Cartan decomposition}
\label{sec:2.3}

For every $k$-group~$\mathbf{G}$, we denote by~$G$ or~$\mathbf{G}_k$ the set of its $k$-points.

Let $\mathbf{G}$ be a connected reductive $k$-group, $\mathbf{Z}$ be the center of~$\mathbf{G}$ and $\mathbf{S}$ be the derived $k$-subgroup of~$\mathbf{G}$, so that $\mathbf{G}$ is the ``almost product'' of $\mathbf{S}$ and~$\mathbf{Z}$. Let $\mathbf{A}$~be a maximal $k$-split torus of~$\mathbf{G}$, $r = r_G$, $r_S$ and $r_Z$ the respective $k$-ranks of $\mathbf{G}$, $\mathbf{S}$ and~$\mathbf{Z}$ so that $r = r_S + r_Z$. Let $X^*(\mathbf{A})$~be the set of characters of~$\mathbf{A}$ (this is a free $\mathbb{Z}$-module of rank~$r$), $E := X^*(\mathbf{A}) \otimes_\mathbb{Z} \mathbb{R}$ and $E_S$ the vector subspace of~$E$ spanned by the characters that are trivial on~$\mathbf{A} \cap \mathbf{Z}$. We call $\Sigma = \Sigma(\mathbf{A}, \mathbf{G})$ the set of roots of~$\mathbf{A}$ in~$\mathbf{G}$: these are the nontrivial weights of~$\mathbf{A}$ in the adjoint representation of the group~$\mathbf{G}$. $\Sigma$~is a root system of~$E_S$ (\cite{Bo-Ti}~\S 5). We choose a system of positive roots~$\Sigma^+$, we call $\Pi = \{\alpha_1, \ldots, \alpha_r\}$ the set of simple roots and we set
\begin{align*}
A^o &:= \setsuch{a \in A}{\forall \chi \in X^*(\mathbf{A}),\; \chi(a) \in k^o}; \\
A^+ &:= \setsuch{a \in A^o}{\forall \chi \in \Sigma^+,\; \chi(a) \in k^+}; \\
A^{++} &:= \setsuch{a \in A^o}{\forall \chi \in \Sigma^+,\; \chi(a) \neq 1}.
\end{align*}
Let $N$ be the normalizer of~$A$ in~$G$, $L$~be the centralizer of~$A$ in~$G$ and $W := N/L$ be the little Weyl group of~$G$: it can be identified with the Weyl group of the root system~$\Sigma$. The subset~$A^+$ is called the positive Weyl chamber. We have the equality $A^o = \bigcup_{w \in W} w A^+$. We endow~$E$ with a $W$-invariant scalar product and we call $(\omega_1, \ldots, \omega_r)$ the fundamental weights of~$\Sigma$. These are the elements of~$X^*(\mathbf{A})$ such that $\frac{2 \langle \omega_i, \alpha_j \rangle}{\langle \alpha_j, \alpha_j \rangle} = \delta_{i, j}$ for all~$i, j$. They form a basis of~$E_S$.

Suppose now that there exists a maximal compact subgroup~$K$ of~$G$ such that $N = (N \cap K) \cdot A$. This assumption is innocuous: it is satisfied when $\mathbf{S}$~is simply connected; we can reduce the problem to this case by standard methods (see \cite{Mar} I.1.5.5 and~I.2.3.1). We then have the equality $G = K A^+ K$, called the \emph{Cartan decomposition} of~$G$. Thus for every $g$ in~$G$, there exists an element~$\mu(g)$ in~$A^+$ such that $g$~is in~$K \mu(g) K$. This element~$\mu(g)$ is unique. We shall call \emph{Cartan projection} this map $\mu: G \to A^+$. This is a continuous and proper map. From now on, every time we mention the Cartan projection~$\mu$, we implicitly assume the existence of such a compact subgroup~$K$.

We call \emph{opposition involution} the map $\imath: A^+ \to A^+$ defined by $\imath(a) := \mu(a^{-1})$. We denote by $\alpha \mapsto \alpha^-$ the permutation of~$\Pi$, also called the opposition involution, defined by $\alpha^-(a) := \alpha(\imath(a))$, for every $a$ in~$A^+$.

\subsection{Jordan decomposition}
\label{sec:2.4}

We will find it convenient to inject the semigroup~$A^+$ into a salient convex cone with nonempty interior~$A^\times$, contained in some $r$-dimensional $\mathbb{R}$-vector space~$A^\bullet$. We will now define a map $\lambda: G \to A^\times$.

When $k = \mathbb{R}$ or~$\mathbb{C}$, we set $A^\bullet := A^o$ and $A^\times := A^+$. The identification of~$A^\bullet$ with its Lie algebra makes it an $\mathbb{R}$-vector space. Every element~$g$ of~$G$ has a unique decomposition, called its Jordan decomposition, into a product $g = g_h g_e g_u$ of three pairwise commuting elements of~$G$, with $g_e$~elliptic, $g_h$~hyperbolic (\ie conjugate to an element~$a(g)$ of~$A^+$) and $g_u$~unipotent. We then simply set $\lambda(g) := a(g)$.

When $k$~is non-Archimedean, $A^o$ is a free $\mathbb{Z}$-module of rank~$r$. We set $A^\bullet := A^o \otimes_\mathbb{Z} \mathbb{R}$ and we define~$A^\times$ to be the convex hull of~$A^+$ in~$A^\bullet$. In this case, a suitable power~$g^n$ of~$g$ has a Jordan decomposition. We then set, using the same notations as for $k = \mathbb{R}$, $\lambda(g) := \frac{1}{n}a(g^n)$. This is an element of~$A^\times$ that does not depend on the choice of~$n$.

The opposition involution $\imath: A^+ \to A^+$ extends to a unique $\mathbb{R}$-linear map, still denoted by~$\imath$, from~$A^\bullet$ to itself that preserves the cone~$A^\times$. For every $g$ in~$G$, we have $\mu(g^{-1}) = \imath(\mu(g))$, $\lambda(g^{-1}) = \imath(\lambda(g))$ and, for~$n \geq 1$, $\lambda(g^n) = n\lambda(g)$. If $\lambda(g) \neq 1$, we call~$L_g$ the half-line of~$A^\times$ that contains~$\lambda(g)$; otherwise we set $L_g := 0$.

For every character~$\chi$ of~$A$, the morphism $|\chi|: A^o \to (0, \infty)$, defined by $|\chi|(a) := |\chi(a)|$, uniquely extends to a group morphism, still denoted by~$|\chi|$, from~$A^\bullet$ to~$(0, \infty)$. For every subset $\theta$ of~$\Pi$, we call $\theta^c$ the complement of~$\theta$ in~$\Pi$, $\theta^- := \setsuch{\alpha^-}{\alpha \in \theta}$, $A^\bullet_\theta := \setsuch{a \in A^\bullet}{\forall \alpha \in \theta^c,\; |\alpha|(a) = 1}$ and $A^\times_\theta := A^\times \cap A^\bullet_\theta$. This is a convex cone in the $\mathbb{R}$-vector space~$A^\bullet_\theta$. We call $A^{\times \times}_\theta$ the relative interior of~$A^\times_\theta$, $A^+_\theta := A^\times_\theta \cap A^+$ and $A^{++}_\theta := A^{\times \times}_\theta \cap A^+$. Thus $A^+_\theta$ (resp.\@~$A^{++}_\theta$) is the closed (resp.\ open) facet of type~$\theta$ of the Weyl chamber~$A^+$. We write~$A^\bullet_i$, $A^\times_i$, ... for $A^\bullet_{{\alpha_i}^c}$, $A^\times_{{\alpha_i}^c}$, ...

\subsection{Representations of~$G$}
\label{sec:2.5}

Let $\rho$~be a representation of~$G$ on a finite-dimensional $k$-vector space~$V$, \ie a $k$-morphism of $k$-groups $\rho: \mathbf{G} \to \mathbf{GL}(V)$. For $\chi$ in~$X^*(\mathbf{A})$, we call
\[V_\chi := \setsuch{v \in V}{\forall a \in A,\; \rho(a) v = \chi(a) v}\]
the corresponding eigenspace. We call $\Sigma(\rho) := \setsuch{\chi \in X^*(\mathbf{A})}{V_\chi \neq 0}$ the set of $k$-weights of~$V$. This set is invariant by the action of the Weyl group~$W$ and we have $V = \bigoplus_{\chi \in \Sigma(\rho)} V_\chi$. We endow~$X^*(\mathbf{A})$ with the order defined by:
\[\chi_1 \leq \chi_2 \iff \chi_2 - \chi_1 \in \sum_{\chi \in \Sigma^+} \mathbb{N} \chi.\]

When $\rho$~is irreducible, the set~$\Sigma(\rho)$ has a unique element~$\chi_\rho$ maximal for this order, called the \emph{highest $k$-weight} of~$V$. We set $\theta_\rho := \setsuch{\alpha \in \Pi}{\chi_\rho - \alpha \in \Sigma(\rho)}$. We will sometimes say that $\theta_\rho$~is the \emph{type} of~$\rho$ or that $\rho$~is \emph{of type}~$\theta_\rho$. For example, the trivial representation is of type~$\emptyset$.

The following two preliminary lemmas are taken from \cite{Be} \S 2.3 and~\S 2.4. They reduce the study of~$\mu(g)$ and of~$\lambda(g)$ to that of~$\|\rho(g)\|$ and of~$\lambda_1(\rho(g))$ for some representations of~$G$.

\begin{lemma}[\cite{Ti1}]
\label{lem_2.5.1}
There exist $r$~irreducible representations~$\rho_i$ of~$G$ on $k$-vector spaces~$V_i$ whose highest $k$-weights~$\chi_i$ are integer multiples of fundamental weights~$\omega_i$ and such that $\dim (V_i)_{\chi_i} = 1$.
\end{lemma}

\begin{remarks*}~
\begin{itemize}
\item When $\mathbf{G}$~is $k$-split, we may simply take $\chi_i = \omega_i$.
\item We fix from now on a family of such representations~$(V_i, \rho_i)$, we choose norms~$\|\cdot\|$ on each of the~$V_i$, we set $X_i := \mathbb{P}(V_i)$, $V^*_i$~the dual of~$V_i$ and $X^-_i := \mathbb{P}(V^*_i)$. We may assume that these choices are made in such a way that, whenever $\alpha_i = \alpha^-_j$, we have $X_i = X^-_j$.
\item We complete this family of $r_S$~representations of~$G$ by $r_Z$~one-dimensional representations still denoted by $(V_i, \rho_i)_{r_S < i \leq r}$ with weights~$\chi_i$ so that the characters~$(\chi_i)_{1 \leq i \leq r}$ form a basis of~$E$.
\end{itemize}
\end{remarks*}

\begin{lemma}
\label{lem_2.5.2}
For every irreducible representation $(V, \rho)$ of~$G$ with highest $k$-weight~$\chi$ and for every norm on~$V$, there exists a constant $C_\chi > 0$ such that, for every $g$ in~$G$, we have
\[C_\chi^{-1} \leq \frac{|\chi(\mu(g))|}{\|\rho(g)\|} \leq C_\chi.\]

Moreover, we have $|\chi|(\lambda(g)) = \lambda_1(\rho(g))$.
\end{lemma}

\begin{remark*}
For $g$ in~$G$, $\rho_i(g)$~is proximal if and only if $\lambda(g)$ is not in~$A^\times_i$. Indeed, the weights of~$A$ in~$V_i$ other than~$\chi_i$ are of the form
\[\chi_i - \alpha_i - \sum_{1 \leq j \leq r} n_j \alpha_j
\quad \text{with } n_j \geq 0.\]
\end{remark*}

Here is an example of application of the two previous lemmas that will be useful to us in~\ref{sec:4.6}.
\begin{corollary*}
For every $g$ in~$G$, we have the equality in~$A^\bullet$: $\displaystyle \lambda(g) = \lim_{n \to \infty} \frac{1}{n}\mu(g^n)$.
\end{corollary*}
\begin{proof}
Every endomorphism~$x$ of a finite-dimensional vector space satisfies the identity $\lambda_1(g) = \displaystyle \lim_{n \to \infty} \|x^n\|^\frac{1}{n}$.

For $x = \rho_i(g)$, we obtain by Lemma~\ref{lem_2.5.2}
\[|\chi_i|(\lambda(g)) = \lim_{n \to \infty} \|\rho_i(g)^n\|^\frac{1}{n} = \lim_{n \to \infty} |\chi_i(\mu(g^n))|^\frac{1}{n} = \lim_{n \to \infty} |\chi_i|(\frac{1}{n}\mu(g^n)).\]

As this holds for every $i = 1, \ldots, r$, we have the desired equality.
\end{proof}

\begin{framed}
\begin{center}
In the whole paper, $k$~is some local field, \\
$\mathbf{G}$~is a connected reductive $k$-group, \\
$\Gamma$~is a Zariski-dense subsemigroup of $G := \mathbf{G}_k$ and \\
we keep the notations introduced in these preliminaries. \\
In sections \ref{sec:5.2},~\ref{sec:6.3} and the following ones, we assume $k = \mathbb{R}$. \\
In section~\ref{sec:5.3}, we assume that $k$~is non-Archimedean.
\end{center}
\end{framed}

\section{The subset~$\theta_\Gamma$ and the limit set~$\Lambda_\Gamma$}
\label{sec:3}

In this section, we associate to~$\Gamma$ a subset~$\theta_\Gamma$ of the set~$\Pi$ of the simple roots of~$G$, or equivalently a flag variety~$Y_\Gamma$ of~$G$ (\ref{sec:3.2}). This variety~$Y_\Gamma$ is the largest variety on which $\Gamma$~acts in a proximal way (\ref{sec:3.5}). When $k = \mathbb{R}$, this subset is equal to~$\Pi$ and $Y_\Gamma$ is always the full flag variety (\ie the one corresponding to minimal parabolics). When $k$~is non-Archimedean, any subset of~$\Pi$ may be obtained in this way.

We also associate to~$\Gamma$ a closed subset~$\Lambda_\Gamma$ of~$Y_\Gamma$ called the limit set of~$\Gamma$. The limit set of the opposite semigroup~$\Gamma^-$ is denoted by~$\Lambda^-_\Gamma$.

We finally associate to~$\Gamma$ a closed set~$F_\Gamma$ of facets of type~$\theta_\Gamma$ called ``quasiperiodic facets'' and we show how to deduce~$F_\Gamma$ from $\Lambda_\Gamma$ and from~$\Lambda^-_\Gamma$ and vice-versa (\ref{sec:3.6}).

\subsection{Simultaneous proximality}
\label{sec:3.1}

We shall need the following lemma.

\begin{lemma*}[\cite{A-M-S}~Lemma~5.15]
Let $W$~be a $k$-vector space, $r: G \to \GL(W)$ a representation that decomposes into a direct sum of irreducible representations $(W, r) = \bigoplus_{1 \leq i \leq l} (W_i, r_i)$.

If for every~$i$, $r_i(\Gamma)$~contains a proximal element, then there exists $\gamma$ in~$\Gamma$ such that for every~$i$, $r_i(\gamma)$~is proximal.
\end{lemma*}

Let us give a proof of this lemma that follows the ideas of~\cite{A-M-S} but simplified thanks to an idea that I found in~\cite{Pr}, which is to introduce the closure of~$r_i(\Gamma)$ in~$\mathbb{P}(\End(W_i))$ rather than in the set of quasiprojective transformations of~$\mathbb{P}(W_i)$.

\begin{proof}
We may assume that $G \subset \GL(W)$. Let
\[S := \setsuch{g \in \GL(W)}{\forall i,\; \restr{g}{W_i} \text{ is a scalar}}.\]
We may assume that $S$~is contained in~$\Gamma$. Let $\Gamma_i := \setsuch{\restr{g}{W_i}}{g \in \Gamma}$, $\overline{\Gamma_i}$~be the closure of~$\Gamma_i$ in~$\End(W_i)$ and $\overline{\Gamma}$~be the closure of~$\Gamma$ in~$\End(W)$. Clearly we have $\overline{\Gamma} \subset \overline{\Gamma_1} \times \cdots \times \overline{\Gamma_l}$. Let
\[\Delta := \setsuch{\pi = (\pi_1, \ldots, \pi_l) \in \overline{\Gamma}}{\forall i,\; \pi_i \neq 0}\]
and let $\pi$ be an element of~$\Delta$ with minimum rank. Let us show that for every~$i$, $\pi_i$~has rank~$1$.

By contradiction, suppose $\rank(\pi_1) \neq 1$. Since $\Gamma_1$~contains some proximal element $h_1 := r_1(\gamma_1)$, $\overline{\Gamma_1}$~contains some projector~$\sigma_1$ with rank~$1$: $\sigma_1 = \lim_{n \to \infty} c_n h_1^n$ for some well-chosen sequence~$c_n$ in~$k$. Since $S$~is contained in~$\Gamma$, there exists an element $\sigma = (\sigma_1, \ldots, \sigma_l)$ of~$\overline{\Gamma}$ such that for every~$i$, $\sigma_i$~is nonzero and such that $\sigma_1$~is the projector defined above. By irreducibility of~$W_i$ and Zariski-connectedness of~$\Gamma$, we may find an element $\gamma$ of~$\Gamma$ such that for every~$i$, $\gamma(\Image(\pi_1)) \not\subset \Ker(\sigma_i)$. Then $\sigma \gamma \pi$~is in~$\Delta$ and has rank smaller than~$\gamma$. Contradiction. Hence for every~$i$, $\rank(\pi_i) = 1$.

Replacing if needed $\pi$ by~$\gamma \pi$ with $\gamma$~an element of~$\Gamma$ such that for every~$i$, $\gamma(\Im(\pi_i)) \not\subset \Ker(\pi_i)$, we may assume that for every~$i$, $\pi_i$~is a multiple of some rank-$1$ projector. Let us then choose a sequence $\gamma_n$ in~$\Gamma$ such that $\lim_{n \to \infty} \gamma_n = \pi$. We then have, for every~$i$, $\lim_{n \to \infty} r_i(\gamma_n) = \pi_i$ and, for $n \gg 0$, $r_i(\gamma_n)$~is proximal.
\end{proof}

\begin{corollary_global}
\label{cor_3.1.1}
We keep the same notations. Then the set
\[\Gamma' := \setsuch{\gamma \in \Gamma}{\forall i,\; r_i(\gamma) \text{ is proximal}}\]
is still Zariski-dense in~$G$.
\end{corollary_global}

\begin{proof}
Let $\gamma_0$~be an element of~$\Gamma'$. For $i = 1, \ldots, l$, we call~$c_i$ the eigenvalue of~$r_i(\gamma_0)$ such that $|c_i| = \lambda_1(r_i(\gamma_0))$. The limit $\pi_i := \lim_{m \to \infty} c_i^{-m} r_i(\gamma_0^m)$ is a projector of rank~$1$. We introduce the Zariski-open subset of~$G$
\[U := \setsuch{g \in G}{\forall i = 1, \ldots, l,\quad g(\Image(\pi_i)) \not\subset \Ker(\pi_i)}.\]
This open subset~$U$ is also Zariski-dense in~$G$, since the representations~$W_i$ are irreducible.

Let us first show that, for every~$\gamma$ in~$U \cap \Gamma$, there exists~$m_0$ such that for $m \geq m_0$, $\gamma_0^m \gamma$~is in~$\Gamma'$. Indeed, the limit
\[\lim_{m \to \infty} c_i^{-m} r_i(\gamma_0^m \gamma) = \pi_i r_i(\gamma)\]
is a nonzero multiple of a rank-$1$ projector and $r_i(\gamma_0^m \gamma)$~is proximal for $m \gg 0$.

Let us now show that $\Gamma'$~is Zariski-dense in~$G$. Let $F$~be a Zariski-closed subset containing~$\Gamma'$; let us show that $F$ contains~$\Gamma$. If suffices to show that $F$ contains~$U \cap \Gamma$. So let $\gamma$~be in~$U \cap \Gamma$; then there exists an integer~$m_0$ such that for every~$m \geq m_0$, $\gamma_0^m \gamma$~is in~$\Gamma'$. Hence
\[\gamma_0^m \gamma \in F,\quad \forall m \geq m_0.\]
Since any Zariski-closed semigroup is a group, this statement still holds for $m$ in~$\mathbb{Z}$. In particular, for $m = 0$, we have $\gamma \in F$. This is what we wanted.
\end{proof}

For every element~$\gamma$ of~$\Gamma'$, we call~$x_{i, \gamma}$ the attracting point of~$r_i(\gamma)$ in~$\mathbb{P}(W_i)$.

\begin{corollary_global}
\label{cor_3.1.2}
We keep the same notations. Let $\eps > 0$ and $\gamma_0$~be an element of~$\Gamma'$. Then the set
\[\Gamma'_\eps := \setsuch{\gamma \in \Gamma'}{\forall i,\; d(x_{i, \gamma},\; x_{i, \gamma_0}) < \eps}\]
is still Zariski-dense in~$G$.
\end{corollary_global}
\begin{proof}
Replace~$\Gamma'$ by~$\Gamma'_\eps$ in the proof of Corollary~\ref{cor_3.1.1}.
\end{proof}

The following corollary will be useful to us in~\ref{sec:4.5}.

\begin{corollary_global}[\cite{A-M-S}~Theorem~5.17]
\label{cor_3.1.3}
With the same notations. There exists a finite subset~$F$ of~$\Gamma$ and $\eps > 0$ such that, for every $g$ in~$\Gamma$, there exists $f$ in~$F$ such that for every~$i = 1, \ldots, l$, $r_i(g f)$~is $\eps$-proximal.
\end{corollary_global}

For the proof of this corollary (that does not use quasiprojective transformations), see~\cite{A-M-S}.

\subsection{The subset~$\theta_\Gamma$}
\label{sec:3.2}

\begin{definition*}
Let $g$ be an element of~$G$. We call~$\theta_g$ the subset of~$\Pi$ such that $\lambda(g) \in A^{\times \times}_{\theta_g}$.
\end{definition*}

In other terms, $\theta_g = \setsuch{\alpha_i \in \Pi}{\rho_i(g) \text{ is proximal}}$ (cf. the last remark of~\ref{sec:2.5}).

We will sometimes say that $\theta_g$~is the ``type'' of~$g$ or that $g$~is of type~$\theta_g$.

\begin{definition*}
Let $\theta$~be a subset of~$\Pi$. We say that an element~$g$ of~$G$ is \emph{$\theta$-proximal} (or \emph{proximal on~$Y_\theta$}) if it satisfies one of the following three equivalent properties:
\begin{itemize}
\item $\theta_g \supset \theta$;
\item for every~$\alpha_i$ in~$\theta$, $\lambda(g) \not\in A^\times_i$;
\item for every~$\alpha_i$ in~$\theta$, the element~$\rho_i(g)$ is proximal.
\end{itemize}
\end{definition*}

When $\theta = \Pi$, we also say that $g$~is $k$-regular.

\begin{definition*}
We say that an element~$g$ of~$G$ is \emph{$(\theta, \eps)$-proximal} if for every~$\alpha_i$ in~$\theta$, the element $\rho_i(g)$~is $\eps$-proximal in~$\mathbb{P}(V_i)$.
\end{definition*}

\begin{definition*}
We call $\theta_\Gamma$ the smallest subset~$\theta$ of~$\Pi$ such that $\lambda(\Gamma) \subset A^\times_\theta$.
\end{definition*}

In other terms, $\theta_\Gamma$ is the union of all the subsets~$\theta_g$ for $g \in \Gamma$.

We shall sometimes say that $\theta_\Gamma$~is the ``type'' of~$\Gamma$ or that $\Gamma$~is of type~$\theta_\Gamma$.

\begin{remarks*}~
\begin{itemize}
\item We set $\Gamma^- := \setsuch{g^{-1}}{g \in \Gamma}$. This is also a semigroup, and we have the identity: $\theta_{\Gamma^-} = \theta_\Gamma^-$. In particular, if $\Gamma$~is a group, then $\theta_\Gamma$~is stable by the opposition involution.
\item It follows from the appendix of~\cite{Be-La} (see also \cite{Go-Ma}, \cite{Gu-Ra} and~\cite{Pr}) that when $k = \mathbb{R}$, and $\Gamma$~is a Zariski-dense subsemigroup of~$G$, then $\theta_\Gamma = \Pi$.
\item The case $k = \mathbb{C}$ is not very interesting for our problem: indeed, if $\Gamma$~is a Zariski-dense (over~$\mathbb{C}$) subsemigroup of~$G$, a restriction of scalars from~$\mathbb{C}$ to~$\mathbb{R}$ allows one to consider~$\Gamma$ as a Zariski-dense (over~$\mathbb{R}$) subsemigoup in the group of real points of some semisimple $\mathbb{R}$-group. For instance, $\Gamma = \SU(n, \mathbb{R})$~is Zariski-dense (over~$\mathbb{C}$) in~$G = \SL(n, \mathbb{C})$.
\end{itemize}
\end{remarks*}

\begin{proposition*}
The set of~$\theta_\Gamma$-proximal elements of~$\Gamma$ is still Zariski-dense in~$G$.
\end{proposition*}
\begin{proof}
This follows from the definitions and from Corollary~\ref{cor_3.1.1} in~\ref{sec:3.1}.
\end{proof}

The following lemma gives other possible equivalent definitions for~$\theta_\Gamma$.

\begin{lemma*}
Let $\alpha_i \in \Pi$. Then the following are equivalent:
\begin{enumerate}[label=\roman*)]
\item $\alpha_i \in \theta_\Gamma$;
\item $\rho_i(\Gamma)$ contains proximal elements;
\item the set~$\alpha_i(\mu(\Gamma))$ is unbounded (in~$k$).
\end{enumerate}
\end{lemma*}

\begin{proof}~

\textit{i)~$\iff$~ii)~} This results from the last remark of~\ref{sec:2.5}.

\textit{ii)~$\implies$~iii)~} Let $g$~be an element of~$\Gamma$ such that $\lambda(g)$~is not in~$A^\times_i$. Since $\lambda(g) = \lim_{n \to \infty} \frac{1}{n} \mu(g^n)$ (by Corollary~\ref{sec:2.5}), the sequence~$\alpha_i(\mu(g^n))$ is unbounded.

\textit{iii)~$\implies$~ii)~} Let $g_m$~be a sequence in~$\Gamma$ such that $\lim_{m \to \infty} |\alpha_i(\mu(g_m))| = +\infty$. Let us write $g_m = k_{1, m} a_m k_{2, m}$ with $k_{1, m}$ and~$k_{2, m}$ in~$K$ and $a_m$~in~$A^+$. We may suppose that the sequences $k_{1, m}$ and~$k_{2, m}$ converge to $k_1$ and~$k_2$ respectively. By assumption, there exists a sequence of constants~$c_m$ in~$k$ such that the limit $\pi_i = \lim_{m \to \infty} c_m^{-1} \rho_i(a_m)$ exists and is a projector of rank~$1$. Let us choose an element~$h$ of~$\Gamma$ such that $\rho_i(h k_1) (\Image \pi_i) \not\subset \rho_i(k_2^{-1})(\Ker \pi_i)$. Such an element~$h$ exists since $\Gamma$~is Zariski-dense in~$G$ and the representation~$\rho_i$ is irreducible. But then the limit $\lim_{m \to \infty} c_m^{-1} \rho_i(h g_m) = \rho_i(h k_1) \pi_i \rho_i(k_2)$ is a nonzero multiple of a rank-$1$ projector. Hence, for $m \gg 0$, $\rho_i(h g_m)$~is proximal.
\end{proof}

\begin{corollary*}
We have the equivalence:
\[\Gamma \text{ is bounded modulo the center of } G \quad\iff\quad \theta_\Gamma = \emptyset.\]
\end{corollary*}

\begin{remark*}
Such a situation can not occur when $k$~is equal to~$\mathbb{R}$ and $G/Z$~is not compact.
\end{remark*}

\begin{proof}
Indeed, we have the equivalences: $\Gamma$~is bounded modulo~$Z$ $\iff$ $\mu(\Gamma)$~is bounded modulo~$Z$ $\iff$ $\forall \alpha \in \Pi,\;$ $\alpha(\mu(\Gamma))$~is bounded $\iff$ $\theta_\Gamma = \emptyset$.
\end{proof}

\subsection{Flag varieties~$Y_\theta$ and varieties~$Z_\theta$}
\label{sec:3.3}

The goal of this subsection is to introduce a few fairly classical notations. For every subset~$\theta$ of~$\Pi$, we call:
\begin{itemize}
\item $\langle \theta \rangle$ the set of all elements of~$\Sigma^+$ that are linear combinations of elements of~$\theta$;
\item $\mathfrak{u}_\theta$ (resp.\@~$\mathfrak{u}^-_\theta$) the Lie algebra obtained as a direct sum of the root spaces~$\mathfrak{g}_\chi$ (resp.\@~$\mathfrak{g}_{-\chi}$) associated to the roots~$\chi$ lying in~$\Sigma^+ \setminus \langle \theta^c \rangle$;
\item $\mathbf{U}_\theta$ (resp.\@~$\mathbf{U}^-_\theta$) the unique unipotent $k$-subgroup normalized by~$\mathbf{A}$ and with Lie algebra~$\mathfrak{u}_\theta$ (resp.\@~$\mathfrak{u}^-_\theta$);
\item $\mathbf{A}_\theta$ the Zariski-connected component of~$\bigcap_{\chi \in \theta^c} \Ker(\chi)$, and
\item $A^+_\theta = \mathbf{A}_\theta \cap A^+$ the standard facet of type~$\theta$; this is consistent with the notations of~\ref{sec:2.4}. We call
\item $\mathbf{L}_\theta$~the centralizer in~$\mathbf{G}$ of~$\mathbf{A}_\theta$;
\item $\mathbf{P}_\theta = \mathbf{L}_\theta \mathbf{U}_\theta$ the standard parabolic $k$-subgroup associated to~$\theta$;
\item $\mathbf{P}^-_\theta = \mathbf{L}_\theta \mathbf{U}^-_\theta$ the standard parabolic $k$-subgroup opposite to~$\mathbf{P}_\theta$;
\item $P_\theta = L_\theta U_\theta$ the group of $k$-points of~$\mathbf{P}_\theta$, called the standard parabolic subgroup of~$G$ associated to~$\theta$, and
\item $P^-_\theta$ the group of $k$-points of~$\mathbf{P}^-_\theta$. We have $P_\emptyset = G$ and $P_\Pi$~is the standard minimal parabolic subgroup of~$G$. Any subgroup~$P$ of~$G$ that is conjugate to~$P_\theta$ is called a parabolic subgroup of~$G$ of type~$\theta$. We call
\item $Y_\theta$ or~$Y^+_\theta$ the set of all parabolic subgroups of type~$\theta$ in~$G$: this is a compact $k$-analytic manifold that identifies to~$G/P_\theta$, on which the group~$K$ acts transitively (cf. \cite{He} and~\cite{Mac}). We call it the \emph{flag variety of type~$\theta$}. We call
\item $Y^-_\theta := Y_{\theta^-}$ the flag variety of type~$\theta^-$; it identifies to~$G/P^-_\theta$. We call
\item $\nu_\theta$ the $K$-invariant probability measure on~$Y_\theta$.
\end{itemize}

These choices have been made so as to make the dimension of~$Y_\theta$ grow with~$\theta$.

~

A subset~$f$ of~$G$ is called a \emph{facet of type~$\theta$} if it is conjugate to~$A^+_\theta$, \ie if we have $f = g A^+_\theta g^{-1}$ with $g$ in~$G$. The subset $g A^{++}_\theta g^{-1}$ is then called the \emph{interior} of the facet~$f$. Every element~$h$ of the interior of the facet~$f$ acts on~$Y_\theta$ with a unique attracting point~$y^+_f$ and on~$Y^-_\theta$ with a unique repelling point~$y^-_f$. These points do not depend on the choice of~$h$ in the interior of the facet~$f$ and we have $y^+_f = g P_\theta g^{-1}$ and~$y^-_f = g P^-_\theta g^{-1}$.

We call $Z_\theta$ the set of facets of type~$\theta$. This is a $k$-analytic submanifold that identifies to~$G/L_\theta$. For example the manifold~$Z_\Pi$ is the set of the Weyl chambers of~$G$.

The injection
\[\fundef{}{Z_\theta}{Y_\theta \times Y^-_\theta}
{f}{(y^+_f, y^-_f)}\]
induces a bijection between~$Z_\theta$ and the open orbit of~$G$ in~$Y_\theta \times Y^-_\theta$. A point~$(y, y^-)$ of~$Y_\theta \times Y^-_\theta$ is said to be \emph{in general position} if it lies in this open orbit. For $y^-$ in~$Y^-_\theta$, we call~$U_{y^-}$ the Zariski-open subset of~$Y_\theta$ formed by points~$y$ such that $(y, y^-)$~is in general position.

For every element~$g$ of~$G$ of type~$\theta$, we call $f_g \in Z_\theta$ the unique facet of type~$\theta$ containing the hyperbolic part of the Jordan decomposition of a power $g^n$ with $n \geq 1$. This is also the unique facet of type~$\theta$ such that $y^+_{f_g}$ is the attracting fixed point for the action of~$g$ on~$Y_\theta$ and $y^-_{f_g}$~is the repelling fixed point for the action of~$g$ on~$Y^-_\theta$. We write $y^+_g$ and~$y^-_g$ instead of $y^+_{f_g}$ and $y^-_{f_g}$.

\subsection{Varieties $Y_\theta$ and~$Z_\theta$ and representations~$\rho_i$}
\label{sec:3.4}

In this subsection, we link these definitions to the representations~$\rho_i$.

For $y$ in~$Y_\theta$ and $\alpha_i$ in~$\theta$, we call $x_i^+(y)$ the line in~$V_i$ invariant by~$y$. The map
\[\fundef{j_\theta:}{Y_\theta}{\prod_{\alpha_i \in \theta} X_i}
{y}{j_\theta(y) := (x^+_i(y))_{\alpha_i \in \theta}}\]
is an embedding whose image is a closed subvariety of~$\prod_{\alpha_i \in \theta} X_i$. We endow this product with the $\sup$~distance and $Y_\theta$ with the distance~$d$ induced by this embedding.

An element~$x^-_i$ of~$X^-_i$ is also a hyperplane of~$V_i$. For $x_i$ in~$X_i$ and $x^-_i$ in~$X^-_i$, we call $\pi_i := x_i \otimes x^-_i \in \mathbb{P}(\End V_i)$ the line containing the endomorphisms of rank~$1$ with image~$x_i$ and with kernel~$x^-_i$. The pair~$(x_i, x^-_i)$ is said to be \emph{in general position} if $x_i$~is not in the hyperplane~$x^-_i$. In this case, the line~$\pi_i$ contains a special element that we still denote by $\pi_i := x_i \otimes x^-_i$: this is the projector with kernel~$x^-_i$ and with image~$x_i$.

For $y^-$ in~$Y^-_\theta$ and $\alpha_i$ in~$\theta$, we call $x^-_i(y^-) \in X^-_i$ the line in~$V^*_i$ invariant by~$y^-$. Let us state that a pair~$(y, y^-)$ in~$Y_\theta \times Y^-_\theta$ is in general position if and only if for every $\alpha_i$ in~$\theta$, the pair of lines~$(x^+_i(y), x^-_i(y^-))$ is in general position.

Let $f$~be a facet of type~$\theta$ and $\alpha_i$ in~$\theta$. We set $x^+_{i, f} := x^+_i(y^+_f)$, $x^-_{i, f} := x^-_i(y^-_f)$ and we call $\pi_{i, f}$ the projector~$x^+_{i, f} \otimes x^-_{i, f}$. The map
\[\fundef{\pi_\theta:}{Z_\theta}{\prod_{\alpha_i \in \theta} \End V_i}
{f}{\pi_\theta(f) := (\pi_{i, f})_{\alpha_i \in \theta}}\]
is once again an embedding whose image is a closed subvariety of~$\prod_{\alpha_i \in \theta} \End V_i$. We endow this product with the $\sup$ norm and $Z_\theta$~with the distance~$d$ induced by this embedding.

For every element~$g$ of type~$\theta$ such that $f_g = f$ and for every $\alpha_i$ in~$\theta$, we have
\[\pi_{i, f} = \lim_{n \to \infty} c_i^{-n} \rho_i(g)^n,\]
where $c_i$~is the eigenvalue of~$\rho_i(g)$ with the largest modulus.

\subsection{The flag variety~$Y_\Gamma$}
\label{sec:3.5}

In this subsection, we link the subset~$\theta_\Gamma$ to the proximality of the action of~$\Gamma$ on the flag varieties and on the projective spaces~$\mathbb{P}(V)$ of various irreducible representations of~$G$.

\begin{definition*}
A sequence $(g_n)_{n \geq 0}$ in~$G$ is said to be \emph{contracting in~$Y_\theta$} if $\lim_{n \to \infty} {g_n}_*(\nu_\theta)$~is a Dirac mass $\delta_{y^+}$ at a point~$y^+$ of~$Y_\theta$, called the \emph{limit point} of the sequence.

We say that a subsemigroup~$\Gamma$ of~$G$ \emph{has the contraction property in~$Y_\theta$} if there exists a sequence $(g_n)_{n \geq 0}$ in~$\Gamma$ that is contracting in~$Y_\theta$. We call a \emph{limit point of $\Gamma$ in $Y_\theta$} the limit point of any such sequence.
\end{definition*}

\begin{proposition*}
Let $\theta$~be a subset of~$\Pi$ and $(\rho, V)$~be an irreducible representation of~$G$ of type~$\theta$ whose highest weight~$\chi_\rho$ has multiplicity~$1$. Then the following statements are equivalent:
\begin{enumerate}[label=\roman*)]
\item $\Gamma$ has the contraction property in~$Y_\theta$;
\item $\theta_\Gamma \supset \theta$;
\item $\rho(\Gamma)$ has the contraction property in~$\mathbb{P}(V)$.
\end{enumerate}
\end{proposition*}

In particular, $Y_{\theta_\Gamma}$~is the ``largest'' flag variety of~$G$ on which $\Gamma$~has the contraction property. We shall write $Y_\Gamma = Y_{\theta_\Gamma}$, $Y^-_\Gamma = Y^-_{\theta_\Gamma}$, $P_\Gamma = P_{\theta_\Gamma}$, $P^-_\Gamma = P^-_{\theta_\Gamma}$ and~$Z_\Gamma = Z_{\theta_\Gamma}$.

\begin{proof}
This results from the following more precise lemma.
\end{proof}

\begin{lemma*}
With the same notations. Let $(g_n)_{n \geq 1}$ be a sequence in~$G$. Consider the following statements:
\begin{enumerate}[label=(\arabic*)]
\item For every $\alpha_i$ in~$\theta$, $\lim_{n \to \infty} |\alpha_i(\mu(g_n))| = \infty$.
\item The sequence~$g_n$ is contracting in~$Y_\theta$ to some limit point~$y^+$.
\item For every $\alpha_i$ in~$\theta$, the sequence~$\rho_i(g_n)$ is contracting in~$\mathbb{P}(V_i)$ to some limit point~$x_i$.
\item The sequence $\rho(g_n)$ is contracting in~$\mathbb{P}(V)$ to some limit point $x_\rho^+$.
\item There exists a pair $(y^+, y^-) \in Y_\theta \times Y^-_\theta$ such that, for every $z$ in the Zariski-open subset $U_{y^-}$ of~$Y_\theta$, we have $\lim_{n \to \infty} g_n z = y^+$. This convergence being uniform on compact subsets of~$U_{y^-}$. %TODO "with the condition that..." or "is then automatically..."
\item For every $\alpha_i$ in~$\theta$, there exist constants~$c_{i, n} \in k^*$ such that the limit $\pi_i = \lim_{n \to \infty} c_{i, n}^{-1} \rho_i(g_n)$ exists and is a rank-$1$ operator.
\item There exist constants $c_n \in k^*$ such that the limit $\pi = \lim_{n \to \infty} c_n^{-1} \rho(g_n)$ exists and is a rank-$1$ operator.
\end{enumerate}

We have the following implications and equivalences:
\[(7) \iff (6) \iff (5) \implies (4) \iff (3) \iff (2) \implies (1).\]
Moreover from any sequence~$(g_n)_{n \geq 1}$ that satisfies~(1), we may extract a subsequence that satisfies~(7).

When (2) holds, we have the identity $x_i = x^+_i(y^+)$.
\end{lemma*}

\begin{remark*}
This gives us another characterization of the elements of~$\theta_\Gamma$ when $\Gamma$~is Zariski-dense: they are the elements~$\alpha_i$ for which $\rho_i(\Gamma)$ has the contraction property in~$\mathbb{P}(V_i)$.
\end{remark*}

\begin{proof}
We may assume that $\mathbf{G}$~is simply connected. The Cartan decomposition of~$G$ allows us to write $g_n = k_{1, n} a_n k_{2, n}$ with $k_{1, n} \in K$, $a_n \in A^+$ and~$k_{2, n} \in K$.

\underline{1st case}: $k_{1, n} = k_{2, n} = 1$. The sequence $g_n = a_n$ takes its values in~$A^+$. In this case, the seven statements are equivalent. We write $y^+_0 := P_\theta$ and $y^-_0 := P^-_\theta$; these are points in~$Y_\theta$ and in~$Y^-_\theta$.

(1) $\iff$ (2) $\iff$ (5). To prove this, we remark (see \cite{Bo-Ti}) that some dense open subset of~$G/P_\theta$ can be identified, as a variety, to the product of the root spaces~$\mathfrak{g}_{-\chi}$ for $\chi$ in $\Sigma^+ \setminus \langle \theta^c \rangle$; that, interpreted in this chart, the action of~$A^+$ is a product action of the actions on each of the factors~$\mathfrak{g}_{-\chi}$; that on these factors, the action of~$A^+$ corresponds to the adjoint action; and finally, that the roots~$\chi$ of $\Sigma^+ \setminus \langle \theta^c \rangle$ are those that can be written $\chi = \sum_{\alpha \in \Pi} n_\alpha \alpha$ with $\sum_{\alpha \in \theta} n_\alpha \geq 1$. In this case, the point~$y_0^+$ is the limit point of~$a_n$ in~$Y_\theta$ and the open set~$U_{y_0^-}$ is its basin of attraction.

(1) $\iff$ (3) $\iff$ (6). We call $x^+_{i,0} \in \mathbb{P}(V_i)$ the highest weight line, $x^-_{i,0} \in \mathbb{P}(V_i^*)$ the hyperplane of~$V_i$ that is the direct sum of the other weight spaces, and $\pi_{i,0} := x^+_{i,0} \otimes x^-_{i,0}$ the projection onto~$x^+_{i,0}$ parallel to~$x^-_{i,0}$. The equivalence results from the fact that the eigenvalue of~$\rho_i(a_n)$ in~$x^+_{i,0}$ is $\chi_i(a_n)$ while the eigenvalue of~$\rho_i(a_n)$ in~$x^-_{i,0}$ that has the largest modulus is $\frac{\chi_i(a_n)}{\alpha_i(a_n)}$ (cf.~\ref{sec:2.5}). In this case, the point~$x^+_{i,0}$ is the limit point of~$\rho_i(a_n)$ in~$\mathbb{P}(V_i)$ and $\pi_{i,0} = \lim_{n \to \infty} (\chi_i(a_n))^{-1} \rho_i(a_n)$.

(1) $\iff$ (4) $\iff$ (7). We proceed in the same fashion.

\underline{2nd case}: The sequences $k_{1,n}$ and~$k_{2,n}$ converge to some elements $k_1$ and~$k_2$. In this case, the seven %TODO Original says "five". It would be nice to check that all seven are indeed equivalent.
statements are still equivalent and we have the equalities: $y^+ = k_1 y^+_0$, $y^- = k_2^{-1} y^-_0$, $x^+_i = \rho_i(k_1) x^+_{i,0}$, $x^-_i = \rho_i(k_2^{-1}) x^-_{i,0}$ and $\pi_i = \rho_i(k_i) \circ \pi_{i,0} \circ \rho_i(k_2^{-1}) = x^+_i \otimes x^-_i$.

\underline{3rd case}: The general case. We recall that from every sequence $k_n$ in~$K$ we can extract a convergent subsequence and that, in a compact set, a sequence is convergent if and only if it has a unique accumulation point. This allows us to deduce from the study of the previous cases the following equivalences:
\begin{align*}
(2) &\iff (1) \text{ and the sequence } k_{1, n}y_0^+ \text{ converges} \\
    &\iff (1) \text{ and } \forall \alpha_i \in \theta, \text{ the sequence } \rho_i(k_{1, n})x_{i, 0} \text{ converges} \\
    &\iff (3).
\end{align*}
Similarly, we have the equivalences:
\begin{align*}
(5) &\iff (1) \text{ and the sequences } k_{1, n}y^+_0 \text{ and } k_{2, n}^{-1}y^-_0 \text{ converge} \\
    &\iff (1) \text{ and } \forall \alpha_i \in \theta, \text{ the sequences } \rho_i(k_{1, n})x_{i, 0} \text{ and } \rho_i(k_{2, n}^{-1})x_{i, 0}^- \text{ converge } \\
    &\iff (6).
\end{align*}
The remaining statements are now clear.
\end{proof}

\subsection{The limit set~$\Lambda_\Gamma$ and the set~$F_\Gamma$ of the quasiperiodic facets}
\label{sec:3.6}

\begin{definition*}
We call \emph{limit set} of~$\Gamma$ the set, denoted by~$\Lambda_\Gamma$ or~$\Lambda_\Gamma^+$, of the limit points of~$\Gamma$ in the flag variety~$Y_\Gamma$; it is a closed subset of~$Y_\Gamma$. We set $\Lambda^-_\Gamma := \Lambda_{\Gamma^-}$; it is a closed subset of~$Y_\Gamma^-$.

We call $F^o_\Gamma$ the set of all facets $f_g \in Z_\Gamma$ corresponding to some $\theta_\Gamma$-proximal element~$g$ of~$\Gamma$; it is a subset of~$Z_\Gamma$. We call $F_\Gamma$ the closure of~$F^o_\Gamma$ in~$Z_\Gamma$. The elements of~$F_\Gamma$ are called \emph{quasiperiodic facets} of~$\Gamma$.
\end{definition*}

\begin{remarks*}~
\begin{enumerate}
\item When $\Gamma$~is a subgroup of $\PSL(2, \mathbb{R})$, the definition of~$\Lambda_\Gamma$ coincides with the classical definition of the limit set as a subset of the boundary at infinity of the Poincaré half-plane $\mathbf{H}$. The set $F^o_\Gamma$ is the set of oriented periodic geodesics on the Riemann surface $\Gamma \backslash \mathbf{H}$.
\item When $\Gamma$~is a subgroup of $\SL(d, \mathbb{R})$, this limit set~$\Lambda_\Gamma$ has been introduced and studied by Y. Guivarc'h \cite{Gu}.
\item It follows from Lemma~\ref{sec:3.5} that, for $\theta$ containing~$\theta_\Gamma$, the set of limit points of~$\Gamma$ in~$Y_\theta$ is nothing else than the image of~$\Lambda_\Gamma$ by the natural projection $Y_\Gamma \to Y_\theta$.
\item When $k = \mathbb{R}$, $F_\Gamma$~is a set of positive Weyl chambers that we call ``quasiperiodic chambers''.
\end{enumerate}
\end{remarks*}

The following lemma generalizes classical properties of the limit sets of subgroups of~$\SL(2, \mathbb{R})$: for example part~iv) says, when $\Gamma$~is a non-elementary subgroup of~$\SL(2, \mathbb{R})$, that for every pair $(y^+, y^-)$ of limit points of~$\Gamma$ on the boundary at infinity, we may find a nontrivial hyperbolic element of~$\Gamma$ whose attracting and repelling points are arbitrarily close to the points $y^+$ and~$y^-$.

\begin{lemma*}~
\begin{enumerate}[label=\roman*)]
\item The limit set~$\Lambda_\Gamma$ is Zariski-dense in~$Y_\Gamma$.
\item Every nonempty $\Gamma$-invariant closed subset~$F$ of~$Y_\Gamma$ contains~$\Lambda_\Gamma$. In particular, the action of~$\Gamma$ on~$\Lambda_\Gamma$ is minimal and no point of~$\Lambda_\Gamma$ is isolated. Moreover, if $\Lambda_\Gamma \neq Y_\Gamma$ then $\Lambda_\Gamma$ has empty interior.
\item For every $y$ in~$\Lambda_\Gamma$ and $\eps > 0$, there exists a $\theta_\Gamma$-proximal element~$g$ of~$\Gamma$ such that $d(y^+_g, y) \leq \eps$. The set of such elements~$g$ is Zariski-dense. In particular, every nonempty open subset of~$\Lambda_\Gamma$ is still Zariski-dense in~$Y_\Gamma$.
\item If we consider $Z_\Gamma$ as a subset of~$Y_\Gamma \times Y^-_\Gamma$ (cf. \ref{sec:3.3}), we have the equality
\[F_\Gamma = Z_\Gamma \cap (\Lambda_\Gamma \times \Lambda^-_\Gamma).\]
In other terms, the set $F^o_\Gamma$ is dense in~$\Lambda_\Gamma \times \Lambda^-_\Gamma$.
\item For every $f$ in~$F_\Gamma$, there exists $\eps_0 > 0$ such that for every $0 < \eps < \eps_0$, the set
\[\Gamma^{(\eps)}_f := \setsuch{g \in \Gamma}{g \text{ is } (\theta_\Gamma, \eps)\text{-proximal and } d(f_g, f) \leq \eps}\]
is still Zariski-dense in~$G$.
\end{enumerate}
\end{lemma*}
\begin{proof}~
\begin{enumerate}[label=\roman*)]
\item This is clear since $\Lambda_\Gamma$ is $\Gamma$-invariant and $\Gamma$~is Zariski-dense in~$G$.
\item Let us show that $F$ contains~$\Lambda_\Gamma$. Let $y$~be a point of~$\Lambda_\Gamma$. Then there exists a sequence $g_n \in \Gamma$ contracting in~$Y_\Gamma$ to a limit point~$y$. By extracting a subsequence if necessary, we may assume that there exists a Zariski-dense open subset~$U_{y^-}$ of~$Y_\Gamma$ such that for every~$y'$ in~$U_{y^-}$, we have $\lim_{n \to \infty} g_n y' = y$. For the same reasons as in~i), the closed set~$F$ is Zariski-dense in~$Y_\Gamma$. Hence we may find a point~$y'$ in the intersection $F \cap U_{y^-}$. The points $g_n y'$ are still in~$F$ and so is~$y$. Hence $F$ contains~$\Lambda_\Gamma$.

In particular, every nonempty $\Gamma$-invariant closed subset of~$\Lambda_\Gamma$ is equal to~$\Lambda_\Gamma$. In other terms, the action of $\Gamma$ on $\Lambda_\Gamma$ is minimal.

Suppose by contradiction that $\Lambda_\Gamma$ contains an isolated point. The set $\overline{\Gamma y} \setminus \Gamma y$ is then a nonempty $\Gamma$-invariant closed subset of~$\Lambda_\Gamma$. Hence it is equal to~$\Lambda_\Gamma$, which is a contradiction.

Moreover, if $\Lambda_\Gamma \neq Y_\Gamma$, the boundary of~$\Lambda_\Gamma$ contains~$\Lambda_\Gamma$, which means that $\Lambda_\Gamma$~has empty interior.

\item Let $g_n \in \Gamma$ be a sequence contracting in~$Y_\Gamma$ to the limit point~$y$. We use Lemma~\ref{sec:3.5}(6). This allows us to construct like in~\ref{sec:3.1} an element~$h$ of~$\Gamma$ and constants $c_{i, n} \in k^*$ such that for every $\alpha_i$ in~$\theta_\Gamma$, the sequence $c_{i, n}^{-1} \rho_i(g_n h)$ converges to a rank-$1$ projector $\pi_i$ whose image is $x_i := x^+_i(y)$. But then, as soon as $n$ is large enough, the element $\rho_i(g_n h)$ is proximal and we have
\[\lim_{n \to \infty} x^+_{\rho_i(g_n h)} = x^+_i(y);\]
and thus the element $g := g_n h$ is $\theta_\Gamma$-proximal and satisfies $d(y_g, y) \leq \eta$.

The second statement then results from Corollary~\ref{cor_3.1.2} in~\ref{sec:3.1}.

The last statement now follows, since every open subset~$V$ of~$\Lambda_\Gamma$ contains a set of the form~$\setsuch{y^+_g}{g \in \Gamma \text{ is } (\theta_\Gamma, \eps)\text{-proximal and } d(y^+_g, y) \leq \eps}$ with $y$ in~$\Lambda_\Gamma$ and $\eps > 0$ sufficiently small. Such a set is Zariski-dense in~$\Lambda_\Gamma$, and so is~$V$.

\item Let $(y^+, y^-)$ be a point of~$\Lambda_\Gamma \times \Lambda^-_\Gamma$. Let us construct an element~$g$ of~$\Gamma$ that is $\theta_\Gamma$-proximal and such that $(y^+_g, y^-_g)$ is close to $(y^+, y^-)$: thanks to~iii), we may find a $\theta_\Gamma$-proximal element~$g_1$ of~$\Gamma$ such that $y^+_{g_1}$~is close to~$y^+$ and $(y^+_{g_1}, y^-)$ is in general position. The same~iii) then allows us to find a $\theta_\Gamma$-proximal element~$g_2$ of~$\Gamma$ such that $y^-_{g_2}$~is close to~$y^-$ and $(y^+_{g_1}, y^-_{g_2})$ is in general position. In particular, using suitably chosen constants, for every $\alpha_i$ in~$\theta_\Gamma$, the limits
\begin{align*}
\pi_{i, 1} &:= \lim_{n \to \infty} c_{i, n}^{-1} \rho_i(g^n_1) \text{ and} \\
\pi_{i, 2} &:= \lim_{n \to \infty} d_{i, n}^{-1} \rho_i(g^n_2)
\end{align*}
are rank-$1$ projectors and the product $\pi_i := \pi_{i, 2} \circ \pi_{i, 1}$ is nonzero.

Since $\Gamma$~is Zariski-dense, we can find an element~$h$ of~$\Gamma$ such that, for every $\alpha_i$ in~$\theta_\Gamma$, the product
\[\pi_{i, 3} := \pi_{i, 1} \circ \rho_i(h) \circ \pi_{i, 2}\]
is nonzero. This operator~$\pi_{i, 3}$ is then some nonzero multiple of the rank-$1$ projector that has the same image as~$\pi_{i, 1}$ and the same kernel as~$\pi_{i, 2}$. The formula
\[\pi_{i, 3} := \lim_{n \to \infty} c_{i, n}^{-1} d_{i, n}^{-1} \rho_i(g_1^n h g_2^n)\]
then proves that, for large enough~$n$, the element $g := g_1^n h g_2^n$ is $\theta_\Gamma$-proximal and the pair $(y^+_g, y^-_g)$ is close to $(y^+_{g_1}, y^-_{g_2})$ and hence also to $(y^+, y^-)$.

\item It suffices to prove it when $f$ is in~$F^o_\Gamma$. We reason as in Corollary~\ref{cor_3.1.1} of~\ref{sec:3.1}. Let $F$~be a Zariski-closed subset of~$G$ containing~$\Gamma^{(\eps)}_f$. Let us prove that $F = G$. Let $g_0$~be a $\theta_\Gamma$-proximal element of~$\Gamma$ such that $f = f_{g_0}$. We now apply the proof of~iv) to $g_1 = g_2 = g_0$. Let $\pi_{i, 0} := \pi_{i, 1} = \pi_{i, 2}$ and $x^+_{i, 0}$~be the image of~$\pi_{i, 0}$ and $X^<_{i, 0} \subset \mathbb{P}(V_i)$~be its kernel. We assume that $\eps_0$~is chosen so that
\[\eps_0 < \frac{1}{2} \inf_{\alpha_i \in \theta_\Gamma} \delta \left( x^+_{i, 0},\; X^<_{i, 0} \right).\]

Then for every $\eps < \eps_0$, there exists a dense Zariski-open subset~$U$ of~$G$ such that for every $h$ in $U \cap \Gamma$, the product $g_0^n h g_0^n$ is in $\Gamma^{(\eps)}_f$ as soon as $n$~is large enough. In particular, this product $g_0^n h g_0^n$ is in~$F$ as soon as $n$~is large enough. Since the Zariski-closure of a semigroup is a group, we deduce that this statement must also hold for~$n = 0$. Hence $h$ is in~$F$. We deduce the inclusion $\Gamma \cap U \subset F$ and the equality $F = G$. This proves that $\Gamma^{(\eps)}_f$ is indeed Zariski-dense in~$G$. \qedhere
\end{enumerate}
\end{proof}

\section{The limit cone~$L_\Gamma$}
\label{sec:4}

The goal of this section is to prove the points \ref{sec:1.2}.a and \ref{sec:1.4}.a of the theorems from the introduction, except for the property ``$L_\Gamma$ has nonempty interior'' that will be proved in Section~\ref{sec:7}.

\begin{definition*}
We call \emph{limit cone} of~$\Gamma$ the smallest closed cone~$L_\Gamma$ in the $\mathbb{R}$-vector space~$A^\times$ that contains the images $\lambda(g)$ of the elements~$g$ of~$\Gamma$.
\end{definition*}

For every subset~$P$ of an $\mathbb{R}$-vector space~$V$, we call \emph{asymptotic cone of~$P$} the set of vectors~$v$ of~$V$ that can be obtained as the limit of a sequence: $v = \lim_{n \to \infty} t_n p_n$ with $t_n > 0$, $\lim_{n \to \infty} t_n = 0$ and $p_n \in P$.

Thus we will prove the following proposition.
\begin{proposition*}~
\begin{enumerate}[label=\alph*)]
\item The limit cone~$L_\Gamma$ is convex;
\item The limit cone~$L_\Gamma$ is the asymptotic cone of the subset~$\mu(\Gamma)$ of~$A^\bullet$.
\end{enumerate}
\end{proposition*}

\subsection{$(\theta, \underline{\eps})$-Schottky groups}
\label{sec:4.1}

The definition below generalizes that $\eps$-Schottky groups introduced in~(\cite{Be} \S~2.2). Let $\theta$~be a subset of~$\pi$ and $\underline{\eps} = (\eps_j)_{j \in J}$ a finite or infinite family of cardinal $t \geq 2$ consisting of positive real numbers.

\begin{definition*}
We say that a Zariski-dense subsemigroup (resp.\ subgroup)~$\Gamma$ of~$G$ with generators $(\gamma_j)_{j \in J}$ is \emph{$(\theta, \underline{\eps})$-Schottky} or ``$\underline{\eps}$-Schottky in~$Y_\theta$'' if for every $\alpha_i$ in~$\theta$, the subsemigroup (resp.\ subgroup)~$\rho_i(\Gamma)$ of~$\GL(V_i)$ generated by $(\rho_i(\gamma_j))_{j \in J}$ is ``$\underline{\eps}$-Schottky in $\mathbb{P}(V_i)$''.

We then write $E_\Gamma := \setsuch{\gamma_j}{j \in J}$ (resp.\ $E_\Gamma := \setsuch{\gamma_j, \gamma_j^{-1}}{j \in J}$) and for every element~$h$ of~$E$ whose index is~$j$, we write $\eps_h = \eps_j$.

When the family~$\underline{\eps}$ is constant equal to some real number~$\eps > 0$, we say that $\Gamma$ is \emph{$(\theta, \eps)$-Schottky}.

When $\theta = \Pi$, we say that $\Gamma$ is \emph{$\underline{\eps}$-Schottky} or \emph{$\eps$-Schottky}.
\end{definition*}

\begin{remarks*}~
\begin{itemize}
\item It is not really useful in this definition to suppose that $\Gamma$ is Zariski-dense. But in practice it will always be.
\item Of course these definitions depend on the choice of the representations~$\rho_i$, of the norms on~$V_i$ and of the generators~$\gamma_j$.
\item If a subsemigroup (resp.\ subgroup)~$\Gamma$ with generators $(\gamma_j)_{j \in J}$ is $(\theta, \underline{\eps})$-Schottky, then so is the subsemigroup (resp.\ subgroup)~$\Gamma_{\underline{m}}$ generated by $(\gamma_j^{m_j})_{j \in J}$, for any family $\underline{m} = (m_j)_{j \in J}$ of positive integers.
\item When $\theta \neq \emptyset$, every \emph{subgroup} $\Gamma$ that is $(\theta, \underline{\eps})$-Schottky on the generators $(\gamma_j)_{j \in J}$ is discrete in~$G$ and is a free group on these generators~$\gamma_j$.
\item Every element of a $(\theta, \underline{\eps})$-Schottky subgroup or subsemigroup is $\theta$-proximal.
\end{itemize}
\end{remarks*}

\begin{definition*}
A word $w = g_l \cdots g_1$ with $g_j$ in~$E_\Gamma$ is said to be \emph{reduced} if $g_{j-1} \neq g_j^{-1}$ for $j = 2, \ldots, l$, and \emph{very reduced} if additionally $g_1 \neq g_l^{-1}$.
\end{definition*}

Of course, in a $(\theta, \underline{\eps})$-Schottky \emph{subsemigroup}, every word is very reduced and in a $(\theta, \underline{\eps})$-Schottky \emph{subgroup}, every word is conjugate to a unique very reduced word.
\begin{definition*}
A subsemigroup~$\Gamma$ of~$G$ is said to be \emph{strongly $(\theta, \eps)$-Schottky} if it is $(\theta, \eps)$-Schottky for the family of generators comprising all elements of~$\Gamma$.
\end{definition*}

\begin{remarks*}~
\begin{itemize}
\item The analogous notion for subgroups is meaningless.
\item Let $\Gamma$~be a strongly $(\theta, \eps)$-Schottky subsemigroup: then $\Gamma \cap \Gamma^- = \emptyset$. On the other hand, every Zariski-dense subsemigroup~$\Gamma'$ of~$\Gamma$ is still strongly $(\theta, \eps)$-Schottky.
\item We shall see that, when $\theta = \Pi$, or when $k$~is non-Archimedean, there always exist strongly $(\theta, \eps)$-Schottky subsemigroups of type~$\theta$ that are open.
\end{itemize}
\end{remarks*}

The following proposition gives an estimate of the projection~$\lambda(w)$ of a word~$w$ of~$\Gamma$. It generalizes Proposition~7.3 from~\cite{Be}.
\begin{lemma*}
For every $\eps > 0$, there exists a compact subset~$M_\eps$ of~$A^\bullet$ satisfying the following.

Let $\theta$~be a subset of~$\Pi$ and $\underline{\eps} = (\eps_j)_{j \in J}$. For every subsemigroup (resp.\ subgroup) of type~$\theta$ that is $(\theta, \underline{\eps})$-Schottky on the generators~$\gamma_1, \ldots, \gamma_t$ and for every very reduced word $w = g_l^{n_l} \cdots g_1^{n_1}$ with $g_j$ in~$E_\Gamma$ and $n_j \geq 1$, we have
\[\lambda(w) - \sum_{1 \leq j \leq l} n_j \lambda(g_j) \;\in\; \left( \sum_{1 \leq j \leq l} M_{\eps_{g_j}} \cap A^\bullet_\theta \right).\]
\end{lemma*}
\begin{proof}
We take $M_\eps = \setsuch{a \in A^\bullet}{\forall i = 1, \ldots, r,\quad C_\eps^{-1} \leq |\chi_i|(a) \leq C_\eps}$ where the~$C_\eps$ are the constants from Lemma~\ref{lem_2.2.2}. Let $C = \prod_{1 \leq j \leq l} C_{\eps_{g_j}}$. Since $\lambda(\Gamma)$ is contained in~$A^\bullet_\theta$ and since for $r_S < i \leq r$ we have $|\chi_i|(\lambda(w)) = |\chi_i|\left(\sum_{1 \leq j \leq l} n_j \lambda(g_j) \right)$, it suffices to show that for every $\alpha_i$ in~$\theta$, we have
\[C^{-1} \;\leq\; |\chi_i|\left( \lambda(w) - \sum_{1 \leq j \leq l} n_j \lambda(g_j) \right) \;\leq\; C.\]
This is a consequence of Lemma~\ref{lem_2.2.2} applied to~$\rho_i(w)$.
\end{proof}

The following corollary is a particular case of the proposition that we want to prove. It is also the key step of its proof.

\begin{corollary*}
If $\Gamma$~is a subsemigroup of type~$\theta$ that is $(\theta, \eps)$-Schottky on the generators~$\gamma_1, \ldots, \gamma_t$, then the convex hull of the half-lines $L_{\gamma_1}, \ldots, L_{\gamma_t}$ is contained in~$L_\Gamma$.
\end{corollary*}
\begin{proof}
Since $L_\Gamma$~is closed, it suffices to show that for any integers $n_1, \ldots, n_t \geq 1$, the element~$n_1 \gamma_1 + \cdots + n_t \gamma_t$ is in~$L_\Gamma$. For $m \geq 1$, we write $w_m := \gamma^{m n_t} \cdots \gamma^{m n_1}$. The previous lemma proves that
\[n_1 \gamma_1 + \cdots + n_t \gamma_t \;\in\; \frac{1}{m} \lambda(w_m) + \frac{t}{m} M_\eps \;\subset\; L_\Gamma + \frac{t}{m} M_\eps.\]
This is true for all $m \geq 1$, hence $n_1 \gamma_1 + \cdots + n_1 \gamma_t$ is in~$L_\Gamma$.
\end{proof}

\subsection{Zariski-density of the elements whose Lyapunov is almost known.}
\label{sec:4.2}

The goal of this subsection is to prove the following lemma:
\begin{lemma*}
Let $\Omega$~be an open cone in~$A^\times$ that meets the limit cone~$L_\Gamma$.

Then the set $\Gamma_\Omega := \setsuch{g \in \Gamma}{\lambda(g) \in \Omega}$ is still Zariski-dense.
\end{lemma*}

We shall prove a more technical result which refines both this lemma and Lemma~\ref{sec:3.6}.v:
\begin{lemmabis*}
We keep the same notations. Let $f \in F_\Gamma$~be a quasiperiodic facet of~$\Gamma$. Recall that $\Gamma^{(\eps)}_f$ is a subset of~$\Gamma$ introduced in Lemma~\ref{sec:3.6}.v.

Then there exists $\eps_0 > 0$ such that for $0 < \eps < \eps_0$, the set $\Gamma^{(\eps)}_{f, \Omega} := \Gamma^{(\eps)}_f \cap \Gamma_\Omega$ is still Zariski-dense in~$G$.
\end{lemmabis*}
\begin{proof}
We may assume that $f$ is in~$F^o_\Gamma$.

Let us first show that $\Gamma^{(\eps)}_{f, \Omega}$ is nonempty. Let us write $\theta = \theta_\Gamma$. Let $g_1$~be a $\theta$-proximal element of~$\Gamma$ such that $f_{g_1}= f$ and $g_2$~be an element of~$\Gamma$ such that $\lambda(g_2)$ is in~$\Omega$. The idea of the proof is to look for an element~$g$ of~$\Gamma^{(\eps)}_{f, \Omega}$ of the form $g_1^r y g_2^s x g_1^r$, where $(x, y)$ is in a suitably chosen Zariski-open subset of $\Gamma \times \Gamma$, $r$~is sufficiently large and $s$ is even larger. Let $i$~be an index such that $\alpha_i \in \theta$ or such that $r_S < i \leq r$. The element $\rho_i(g_1)$ is proximal in~$\mathbb{P}(V_i)$. We call $c_i$ the eigenvalue of~$\rho_i(g_1)$ such that $|c_i| = \lambda_1(\rho_i(g_1))$. The limit $\pi_i := \lim_{r \to \infty} c_i^{-r} \rho_i(g_1^r)$ is a rank-$1$ projector.

Replacing if necessary $g_2$ by some power, we may find an element~$d_i$ in~$k$ such that $|d_i| = \lambda_1(\rho_i(g_2))$. By compactness of~$\mathbb{P}(\End(V_i))$, there exists a sequence~$d_{i, s}$ in~$k$ and a subsequence $S$ of~$\mathbb{N}$ such that the limit $\sigma_i := \lim_{s \in S} d_{i, s}^{-1} \rho_i(g_2^s)$ exists and is nonzero. When $g_2$~is semisimple or when $k$~is non-Archimedean, we may take $d_{i, s} = d_i^s$. In the general case, we only know that, for $s$ in~$S$,
\[\log |d_{i, s}| - s \log |d_i| = \bigo(\log s).\]
We have, when $s$ is in~$S$,
\[\lim_{r, s \to \infty} c_i^{-2r} d_{i, s}^{-1} \rho_i(g_1^r y g_2^s x g_1^r) = \pi_i \rho_i(y) \sigma_i \rho_i(x) \pi_i.\]
Let $U$~be the nonempty Zariski-open subset of~$G \times G$
\[U := \setsuch{(x, y) \in G \times G}{\forall \alpha_i \in \theta,\quad \pi_i \rho_i(y) \sigma_i \rho_i(x) \pi_i \neq 0}.\]
Let us choose $(x, y)$ in~$U \cap (\Gamma \times \Gamma)$. We may write $\pi_i \rho_i(y) \sigma_i \rho_i(x) \pi_i = \beta_i \pi_i$ where $\beta_i$~is some nonzero scalar. Hence there exist $r_0 \geq 0,\; s_0 \geq 0$ such that for $r \geq r_0,\; s \geq s_0$, for $s$ in~$S$ and for $\alpha_i$ in~$\theta$, the element $\rho_i(g_1^r y g_2^s x g_1^r)$ is proximal. We deduce from the previous equality that, writing $g_{r, s} := g_1^r y g_2^s x g_1^r$, we have
\[\lim_{r, s \to \infty} f_{g_{r,s}} = f.\]
We now choose the integers $r_0$ and~$s_0$ such that for $r \geq r_0$ and $s \geq s_0$, we have
\[d(f_{g_{r, s}}, f) \leq \eta.\]
Besides, we have
\[\lim_{r, s \to \infty} |c_i|^{-2r} |d_{i, s}|^{-1} \lambda_1(\rho_i(g_1^r y g_2^s x g_1^r)) = |\beta_i|.\]
Hence, for $\alpha_i$ in~$\theta$ and for $r_S < i \leq r$,
\[\log(|\chi_i|(\lambda(g_1^r y g_2^s x g_1^r) - 2r\lambda(g_1) - s\lambda(g_2))) = \bigo(\log s).\]
Recall that the addition of~$A^\bullet$ that we denote additively is the law that extends the multiplication of~$A^o$ that we denoted multiplicatively.

On the other hand, by definition of~$\theta$, when $\alpha_i$ is in~$\theta^c$, we have
\[|\alpha_i|(\lambda(g_1^r y g_2^s x g_1^r)) = |\alpha_i|(\lambda(g_1)) =|\alpha_i|(\lambda(g_2)) = 1.\]

Since the family $\left( (\log |\alpha_i|)_{\alpha_i \in \theta^c},\; (\log |\chi_i|)_{\alpha_i \in \theta},\; (\log |\chi_i|)_{r_S < i \leq r} \right)$ generates the linear forms on~$A^\bullet$, we deduce the equality in~$A^\bullet$:
\[\lambda(g_1^r y g_2^s x g_1^r) = 2r\lambda(g_1) + s\lambda(g_2) + \bigo(\log s),\]
where $\bigo(\log s)$ now stands for a family of vectors $v_{r, s}$ in~$A^\bullet$ such that $(\log s)^{-1} v_{r, s}$ is bounded. Fix $r \geq r_0$; we then have
\[\lim_{s \in S} \frac{1}{s} \lambda(g_1^r y g_2^s x g_1^r) = \lambda(g_2).\]
Hence for $s$ sufficiently large, $g_1^r y g_2^s x g_1^r$ is in~$\Gamma^{(\eps)}_{f, \Omega}$ and $\Gamma^{(\eps)}_{f, \Omega}$~is nonempty.

Let us now show that $\Gamma^{(\eps)}_{f, \Omega}$~is Zariski-dense. Let $F$~be a Zariski-closed set containing $\Gamma^{(\eps)}_{f, \Omega}$. We can use the same reasoning as in Corollary~\ref{cor_3.1.1} from~\ref{sec:3.1}. Let $g_1$~be in~$\Gamma^{(\eps)}_{f, \Omega}$. For $\alpha_i$ in~$\theta$ and for $r_S < i \leq r$, let once again $c_i$~be the eigenvalue of~$\rho_i(g_1)$ such that $|c_i| = \lambda_1(\rho_i(g_1))$ and $\pi_i$~be the rank-$1$ projector $\pi_i := \lim_{r \to \infty} c_i^{-r} \rho_i(g_1^r)$. Let $U_o$ be the nonempty Zariski-open subset of~$G$
\[U_o := \setsuch{x \in G}{\forall \alpha_i \in \theta,\quad \pi_i \rho_i(x) \pi_i \neq 0}.\]
For $x$ in~$U_o \cap \Gamma$, the same reasoning as above proves that, for large~$r$, the element $g_r := g_1^r x g_1^r$ is $\theta$-proximal and that
\[\lim_{r \to \infty} f_{g_r} = f\]
and
\[\lim_{r \to \infty} \frac{1}{r} \lambda(g_r) = 2\lambda(g_1).\]
Hence there exists $r_0$ such that for $r \geq r_0$, $g_1^r x g_1^r$ is in $\Gamma^{(\eps)}_{f, \Omega} \subset F$. Since the Zariski closure of a semigroup is a group, we deduce that for all $r$ in~$\mathbb{Z}$, $g_1^r x g_1^r$ is in~$F$. In particular, taking $r = 0$, $x$~is in~$F$. Hence $F$ contains $U_o \cap \Gamma$ and $F = G$. This proves that $\Gamma^{(\eps)}_{f, \Omega}$ is indeed Zariski-dense in~$G$.
\end{proof}

\subsection{$(\theta, \eps)$-Schottky subgroups of Zariski-dense groups}
\label{sec:4.3}

To prove our Proposition~\ref{sec:4}.0, we shall need the following proposition only for the case of finite sequences, of length~$2$. The case of infinite sequences will be useful to us in~\ref{sec:5.1}.

\begin{proposition*}
Let $\Gamma$~be a Zariski-dense subsemigroup (resp.\@ subgroup) of~$G$ that is not bounded modulo the center of~$G$. Let $\theta := \theta_\Gamma$ be the type of~$\Gamma$ and $(\Omega_j)_{0 < j < t}$~be a finite or infinite sequence, of length at least~$2$, consisting of open cones of~$A^{\times \times}_\theta$ that meet the cone~$L_\Gamma$.

Then there exists a sequence $\underline{\eps} = (\eps_j)_{0 < j < t}$ and a discrete Zariski-dense subsemigroup (resp.\@ subgroup)~$\Gamma'$ in~$\Gamma$ that is of type~$\theta$ and is $(\theta, \eps)$-Schottky for some generators $(\gamma_j)_{0 < j < t}$ satisfying $\lambda(\gamma_j) \in \Omega_j$ for every~$j$.
\end{proposition*}

\begin{remark*}
The condition that ``$\Gamma$ is not bounded modulo the center of~$G$'' means that the image of~$\Gamma$ in~$G/Z$ is not contained in a compact subgroup. Since $\Gamma$~is Zariski-dense in~$G$, this assumption is equivalent to $\theta_\Gamma \neq \emptyset$. It is automatically satisfied when $k = \mathbb{R}$ and $G/Z$ is not compact.
\end{remark*}

This proposition follows from the following lemma that generalizes Lemma~7.2 from~\cite{Be}.
\begin{lemma*}
We use the same notations as in the proposition.
\begin{enumerate}[label=\alph*)]
\item We may choose elements $(\gamma_j)_{0 < j < t}$ of~$\Gamma$ such that, writing $E = \setsuch{\gamma_j}{0 < j < t}$ (resp.\ $E = \setsuch{\gamma_j}{0 < j < t} \cup \setsuch{\smash{\gamma_j^{-1}}}{0 < j < t}$), we have:
\begin{enumerate}[label=\roman*)]
\item $\lambda(\gamma_j) \in \Omega_j$ for every~$j$. In particular, all the elements of~$E$ are of type~$\theta$.
\item For every $g, h$ in~$E$ (resp.\ for every $g, h$ in~$E$ with $g \neq h^{-1}$), the pair $(y^+_g, y^-_h)$ in $Y_\Gamma \times Y^-_\Gamma$ is in general position.
\item For every $j$, the semigroup generated by~$\gamma_j$ is Zariski-connected.
\item The semigroup generated by $\gamma_1$ and~$\gamma_2$ is Zariski-dense in~$G$.
\item If $t = \infty$, the sequences $y^\pm_{\gamma_j}$ converge in~$Y^\pm_\theta$ to points denoted by~$y^\pm_\infty$. And in this case, condition~ii) is replaced by the following stronger condition: For every $g, h$ in $E \cup \{\infty\}$ (resp.\ for every $g, h$ in $E \cup \{\infty\}$ with $g \neq h^{-1}$), the pair $(y^+_g, y^-_h)$ in $Y_\Gamma \times Y_\Gamma^-$ is in general position.
\end{enumerate}
\item For any such choice, there exist sequences $\underline{\eps} = (\eps_j)_{0 < j < t}$ and $\underline{m^o} = (m^o_j)_{0 < j < t}$ such that, for every sequence of integers $\underline{m} = (m_j)_{0 < j < t}$ with $m_j \geq m^o_j$ for every~$j$, the subsemigroup (resp.\ subgroup) $\Gamma_{\underline{m}}$ generated by $(\gamma_j^{m_j})_{0 < j < t}$ is $(\theta, \eps)$-Schottky, of type~$\theta$, discrete and Zariski-dense in~$G$.
\end{enumerate}
\end{lemma*}

\begin{remarks*}~
\begin{itemize}
\item Recall that, for every $g$ in~$G$, we have the equality in~$Y^-_{\theta_g}$: $y^+_{g^{-1}} = y^-_g$.
\item The symbols $y^\pm_\infty$ are notations introduced for the convenience of being able to state condition~v) in a simple way. They absolutely do not imply that we are considering elements of~$G$ denoted by~$\infty$.
\end{itemize}
\end{remarks*}

\begin{proof}
We will treat the case where $\Gamma$~is a group. We then have $Y^-_\Gamma = Y_\Gamma$ and $\Lambda^-_\Gamma = \Lambda_\Gamma$. The case where $\Gamma$~is a semigroup is rather easier, and is left for the reader.

\begin{enumerate}[label=\alph*)]
\item Since every Zariski-dense subgroup of~$G$ contains a finitely-generated subgroup that is still Zariski-dense, we lose no generality in assuming that $\Gamma$~is finitely generated. Then there exists an integer $n_\Gamma \geq 1$ such that, for every $g$ in~$\Gamma$, the group generated by~$g^{n_\Gamma}$ is Zariski-connected (\cite{Ti2}, Lemma~4.2). Let $G_{\operatorname{reg}}$ stand for the Zariski-open subset of~$G$ formed by regular semisimple elements (\ie elements whose centralizer is a maximal torus of~$G$).

Let $(y^+_\infty, y^-_\infty)$~be a pair in~$\Lambda_\Gamma \times \Lambda_\Gamma$ in general position. Thanks to Lemma~\ref{sec:3.6}~iii), we may choose two sequences $(y^\pm_n)_{n \geq 1}$ in~$\Lambda_\Gamma$ converging to the points $y^\pm_\infty$ and such that, whenever $m, n$ are positive integers or $\infty$ and $\alpha, \beta = \pm$, the pair $(y^\alpha_m, y^\beta_n)$ is in general position except if $m = n$ and $\alpha = \beta$. In the sequel, we will allow ourselves to change this choice of the sequences $y^\pm_n$ but always in such a way as to preserve these conditions.

Let us start by constructing~$\gamma_1$. Thanks to Lemmas~\ref{sec:3.6}~iv) and~\ref{sec:4.2}~bis, we may assume, replacing if necessary the points $y^\pm_1$ by sufficiently close points, that there exists an element $\gamma'_1$ of~$\Gamma$ such that the element $\gamma_1 := (\gamma'_1)^{n_\Gamma}$ is regular semisimple and satisfies $\lambda(\gamma_1) \in \Omega_1$ and $y^\pm_{\gamma_1} = y^\pm_1$.

Since $\gamma_1$~is regular semisimple, the union of the Zariski-closed and Zariski-connected proper subgroups of~$G$ that contain~$\gamma_1$ is contained in a proper Zariski-closed subset $F_{\gamma_1}$ of~$G$ (\cite{Ti2}, Proposition~4.4). We denote by~$F^c_{\gamma_1}$ the complementary Zariski-open subset.

Let us now construct the other~$\gamma_j$ by induction on~$j$. Like previously, thanks to Lemmas~\ref{sec:3.6}~iv) and~\ref{sec:4.2}~bis, we may assume, replacing if necessary the points $y^\pm_j$ by sufficiently close points, that there exists an element $\gamma'_j$ of~$\Gamma$ such that the element $\gamma_j := (\gamma'_j)^{n_\Gamma}$ is in~$F^c_{\gamma_1}$ and satisfies $\lambda(\gamma_j) \in \Omega_j$ and $y^\pm_{\gamma_j} = y^\pm_j$.

By construction, the sequence $\gamma_j$ satisfies i), ii) and~v). Since $\gamma_j$ is the $n_\Gamma$-th power of an element of~$\Gamma$, the semigroup generated by~$\gamma_j$ is Zariski-connected. Since $\gamma_2$ is not in~$F_{\gamma_1}$, the semigroup generated by $\gamma_1$ and~$\gamma_2$ is Zariski-dense in~$G$. This proves ii) and~iv).
\item For every $j < t$, we can find a real number $\eps_j > 0$ such that for every $\alpha_i$ in~$\theta$, $\alpha, \beta = \pm$ and $n < t$ with $\gamma_j^\alpha \neq \gamma_n^\beta$, we have
\begin{align*}
\eps_j &\leq \frac{1}{2} \delta \left( x^+_{\rho_i(\gamma_n^\beta)},\; X^<_{\rho_i(\gamma_j^\alpha)} \right) \text{ and} \\
\eps_j &\leq \frac{1}{2} \delta \left( x^+_{\rho_i(\gamma_j^\alpha)},\; X^<_{\rho_i(\gamma_n^\beta)} \right).
\end{align*}
This is clear if $t < \infty$ because there are only finitely many constraints on~$\eps_j$. If $t = \infty$, this is possible because, since the pairs $(y_\infty^\beta, y_j^{-\alpha})$ are in general position, the sequence of points $\left( x^+_{\rho_i(\gamma_n^\beta)} \right)_{n \geq 1}$ converges in~$\mathbb{P}(V_i)$ to a point that is not in the hyperplane $X^<_{\rho_i(\gamma_j^\alpha)}$ and, similarly, since the pairs $(y_j^\alpha, y_\infty^{-\beta})$ are in general position, the sequence of hyperplanes $\left( X^<_{\rho_i(\gamma_n^\beta)} \right)_{n \geq 1}$ converges to a hyperplane of~$\mathbb{P}(V_i)$ that does not contain the point $x^+_{\rho_i(\gamma_j^\alpha)}$.

We may then find integers $m^o_j$ such that for every $m_j \geq m^o_j$, the elements $\gamma_j^{m_j}$ and~$\gamma_j^{-m_j}$ are $(\theta, \eps_j)$-proximal. The group~$\Gamma_{\underline{m}}$ is then a $(\theta, \underline{\eps})$-Schottky subgroup with generators $(\gamma_j)_{0 < j < t}$. By~iii) the Zariski closure of~$\Gamma_{\underline{m}}$ contains the generators~$\gamma_j$; it is thus equal to~$G$ thanks to~iv). \qedhere
\end{enumerate}
\end{proof}

\subsection{Convexity of the limit cone~$L_\Gamma$}
\label{sec:4.4}

\begin{proof}[Proof of Proposition~\ref{sec:4}.a)]
Let $L_1$ and~$L_2$ be two half-lines in~$L_\Gamma$. By Proposition~\ref{sec:4.3}, we can find a $(\theta, \eps)$-Schottky subsemigroup~$\Gamma'$ of~$\Gamma$ with generators $\gamma_1$, $\gamma_2$ such that the half-lines $L_{\gamma_1}$ and~$L_{\gamma_2}$ are arbitrarily close to $L_1$ and~$L_2$. By Corollary~\ref{sec:4.2}, the convex hull of $L_{\gamma_1}$ and~$L_{\gamma_2}$ is contained in~$L_\Gamma$. Since $L_\Gamma$~is closed, the convex hull of $L_1$ and~$L_2$ is also contained in~$L_\Gamma$. Hence $L_\Gamma$ is convex.
\end{proof}

\subsection{$(\theta, \eps)$-proximal elements}
\label{sec:4.5}

We shall need the following two lemmas:
\begin{lemma*}
There exists a finite subset~$F$ of~$\Gamma$ and a real number $\eps > 0$ with the following property: for every $g$ in~$\Gamma$, we can find an element~$f$ of~$F$ such that the element~$g f$ is $(\theta_\Gamma, \eps)$-proximal.
\end{lemma*}

\begin{proof}
This is a reformulation of Corollary~\ref{cor_3.1.3} from~\ref{sec:3.1} applied to the representations~$\rho_i$, for $\alpha_i$ in~$\theta_\Gamma$.
\end{proof}

The following lemma says that, for a $(\theta, \eps)$-proximal element~$g$, the element~$\lambda(g)$ is a good approximation of the Cartan projection~$\mu(g)$.

\begin{lemma*}
For every $\eps > 0$, there exists a compact subset~$N_\eps$ of~$A^\bullet$ such that for every $(\theta_\Gamma, \eps)$-proximal element~$g$ of~$\Gamma$, we have
\[\mu(g) - \lambda(g) \in N_\eps.\]
\end{lemma*}

\begin{proof}
We set $\theta = \theta_\Gamma$ and we choose, using Lemma~\ref{sec:3.2}, a real number
\[\eps_0 < \left( \sup_{\alpha_i \in \theta^c,\; g \in \Gamma} |\alpha_i(\mu(g))| \right)^{-1}.\]

We set, with the notations of~\ref{lem_2.2.1} and of~\ref{lem_2.5.2}, $C'_\eps = c_\eps^{-1} \sup_{\alpha_i \in \theta} C_{\chi_i}$, and we take
\[N_\eps = \setsuch{a \in A}{\begin{aligned}
&\eps \leq |\alpha_i|(a) \leq \eps^{-1} && \forall \alpha_i \in \theta^c, \\
&C'^{-1}_\eps \leq |\chi_i|(a) \leq C'_\eps && \forall \alpha_i \in \theta \text{ and} \\
&|\chi_i|(a) = 1 && \forall i = r_S+1, \ldots, r.
\end{aligned}}\]
This compact set works. Indeed, on the one hand, for $\alpha_i$ in~$\theta^c$, we have by definition
\[|\alpha_i|(\lambda(g)) = 1 \quad\text{and}\quad 1 \leq |\alpha_i|(\mu(g)) \leq \eps^{-1}.\]
On the other hand, for $\alpha_i$ in~$\theta$, it follows from Lemma~\ref{lem_2.2.1} that
\[c_\eps \leq \frac{\lambda_1(\rho_i(g))}{\|\rho_i(g)\|} \leq 1,\]
and consequently, thanks to Lemma~\ref{lem_2.5.2}, we have
\[c_\eps C^{-1}_{\chi_i} \leq \frac{|\chi_i|(\lambda(g))}{|\chi_i|(\mu(g))} \leq C_{\chi_i}.\]
Finally, for $r_S < i \leq r$ we have
\[|\chi_i|(\lambda(g)) = |\chi_i|(\mu(g)). \qedhere\]
\end{proof}

\subsection{The Cartan projection of~$\Gamma$}
\label{sec:4.6}

The following result is Proposition~5.1 from~\cite{Be}.

\begin{lemma*}
For every compact subset~$L$ of~$G$, there exists a compact subset~$M$ of~$A^\bullet$ such that for every $g$ in~$G$, we have
\[\mu(L g L) \subset \mu(g) + M.\]
\end{lemma*}

Recall that, by definition, the restriction to~$A^+$ of the addition of~$A^\bullet$ coincides with the multiplication of~$A^+$.

The following proposition compares the Cartan projection and the Lyapunov projection of a Zariski-dense subsemigroup.

\begin{proposition*}
There exists a compact subset~$N$ of~$A^\bullet$ such that
\[\mu(\Gamma) \subset \lambda(\Gamma) + N.\]
\end{proposition*}

\begin{remarks*}~
\begin{itemize}
\item This is false if $\Gamma$ is not Zariski-dense: indeed consider as~$\Gamma$ any group generated by a unipotent element.
\item Conversely, the existence of a compact subset~$N'$ of~$A^\bullet$ such that $\lambda(\Gamma) \subset \mu(\Gamma) + N'$ is true for $(\theta, \eps)$-Schottky subgroups and subsemigroups: it suffices to apply Lemma~\ref{sec:4.5}.2 to the very reduced words of~$\Gamma$. I do not know if this existence is true in general.
\end{itemize}
\end{remarks*}

\begin{proof}
Let $F$ be the finite subset of~$\Gamma$ and $\eps$ the constant introduced in Lemma~\ref{sec:4.5}.1. Let $M$~be the compact subset of~$A^\bullet$ defined in the lemma above, corresponding to the subset $L := F^{-1}$.  For every $g$ in~$\Gamma$, we can thus find $f$ in~$F$ such that $g f$~is $(\theta, \eps)$-proximal. Lemma~\ref{sec:4.5}.2 then gives us
\[\mu(g) \;\in\; \mu(g f) + M \;\subset\; \lambda(g f) + M + N_\eps \;\subset\; \lambda(\Gamma) + M + N_\eps.\]
This proves our proposition with $N = M + N_\eps$.
\end{proof}
\begin{proof}[Proof of Proposition~\ref{sec:4}.b.]
Let $C_\Gamma$ stand for the asymptotic cone of~$\mu(\Gamma)$. The inclusion $L_\Gamma \subset C_\Gamma$ follows from Corollary~\ref{sec:2.5} and the reverse inclusion follows from the last proposition.
\end{proof}

\section{Examples}
\label{sec:5}

The goal of this section is to prove points~\ref{sec:1.2}.b and~\ref{sec:1.4}.b of the theorems from the introduction: namely, to construct a Zariski-dense subgroup or subsemigroup whose limit cone is a prescribed convex cone~$\Omega$.

The construction of the subgroup is done in Section~\ref{sec:5.1}: in this setting, we assume that the cone~$\Omega$ has nonempty interior and is invariant by the opposition involution.

The construction of the subsemigroup is done in Sections~\ref{sec:5.2} (for $k = \mathbb{R}$) and~\ref{sec:5.3} (for non-Archimedean~$k$). The semigroup that we construct is open, but it is possible to extract from it a discrete subsemigroup that has the same limit cone: this is explained in Section~\ref{sec:5.1}.

\subsection{Construction of~$\Gamma$}
\label{sec:5.1}

\begin{proposition*}
Let $\Gamma$~be a Zariski-dense subsemigroup (resp.\ subgroup) of~$G$ and $\Omega$ a closed convex cone with nonempty interior in~$L_\Gamma$ (resp.\ a closed convex cone with nonempty interior in~$L_\Gamma$ that is stable by the opposition involution).

We assume that $\Gamma$~is not bounded modulo the center of~$G$. Then there exists a discrete Zariski-dense subsemigroup (resp.\ subgroup)~$\Gamma'$ of~$\Gamma$ such that $L_{\Gamma'} = \Omega$.
\end{proposition*}

\begin{remarks*}~
\begin{itemize}
\item The condition ``$\Omega$ with nonempty interior in~$L_\Gamma$'' means that $\Omega$~is contained in~$L_\Gamma$, that it spans the same vector subspace of~$A^\bullet$ as~$L_\Gamma$ and that this subspace is nonzero.
\item Recall that the condition ``$\Gamma$ not bounded modulo the center of~$G$'' is automatically satisfied when $k = \mathbb{R}$ and $G/Z$ is not compact.
\end{itemize}
\end{remarks*}

\begin{proof}
By assumption the subset $\theta = \theta_\Gamma$ is nonempty. Let $B$~be the vector subspace of~$A^\bullet_\theta$ spanned by~$L_\Gamma$. Let us choose a sequence $(\omega_j)_{j \geq 0}$ of convex cones open in~$B$ and contained in~$\Omega$ such that every cone open in~$B$ and contained in~$\Omega$ is the union of the cones~$\omega_j$ that it contains. Let $\Omega_j$~be some open convex cone of~$A^{\times \times}_\theta$ such that $\Omega_j \cap B = \omega_j$.

Choose a sequence $(\gamma_j)_{j \geq 0}$ in~$\Gamma$ and sequences $(\eps_j)_{j \geq 0}$ and~$(m_j)_{j \geq 0}$ as in Lemma~\ref{sec:4.3}. Let $M_\eps$ be the compact subsets of~$A^\bullet$ introduced in Lemma~\ref{sec:4.1}. We may assume that the compact subsets~$M_\eps$ are invariant by the opposition involution and that the $m_j$ are chosen so that
\[\lambda(\gamma_j^{m_j}) + (M_{\eps_j} \cap B) \subset \Omega.\]

When $\Gamma$~is a group, the invariance of~$\Omega$ by the opposition involution ensures that we also have
\[\lambda(\gamma_j^{-m_j}) + (M_{\eps_j} \cap B) \subset \Omega.\]
Let $\Gamma' = \Gamma_{\underline{m}}$: this is the Zariski-dense $(\theta, \eps)$-Schottky discrete subsemigroup (resp.\ subgroup) with generators $\gamma_j^{m_j}$. Since $\lambda(\gamma_j)$ is in~$\omega_j$, the cone~$L_{\Gamma'}$ intersects all the cones~$\omega_j$, hence $L_{\Gamma'} \supset \Omega$.

Let us prove the opposite inclusion. It suffices to check that for every very reduced word~$w = g_l \cdots g_1$ with every $g_p$ being one of the elements $\gamma_j^{m_j}$ (resp.\@~$\gamma_j^{\pm m_j}$), we have $\lambda(w) \in \Omega$. But Lemma~\ref{sec:4.1}, the previous inclusions and the convexity of~$\Omega$ show that
\[\lambda(w) \;\in\; \sum_{1 \leq p \leq l} (\lambda(g_p) + (M_{\eps_{g_p}} \cap B)) \;\subset\; \Omega.\]
Hence $L_{\Gamma'} = \Omega$.
\end{proof}

\begin{corollary*}
Suppose that $G/Z$ is not compact. Let $\Omega$~be a closed convex cone in~$A^\times$ with nonempty interior (resp.\ a closed convex cone in~$A^\times$ with nonempty interior and invariant by the opposition involution).

Then there exists a Zariski-dense discrete subsemigroup (resp.\ subgroup)~$\Gamma'$ of~$G$ such that $L_{\Gamma'} = \Omega$.
\end{corollary*}
\begin{proof}
Take $\Gamma = G$ in the last proposition.
\end{proof}

\subsection{The open semigroups $G^\eps_f$ and $G^\eps_{f, \Omega}$ for $k = \mathbb{R}$}
\label{sec:5.2}

Let $f$~be an element of~$Z_\Pi$. Recall that, for $i = 1, \ldots, r$, $x^+_{i, f}$~is the fixed point of the parabolic subgroup~$y^+_f$ in~$X_i = \mathbb{P}(V_i)$ and $X^<_{i, f}$~is the unique hyperplane in~$X_i$ invariant by the opposite parabolic subgroup~$y^-_f$. Let $\eps_f := \frac{1}{10} \inf_{1 \leq i \leq r} \delta(x^+_{i, f},\; X^<_{i, f})$. We choose a real number $\eps < \eps_f$ and we define
\[b^\eps_{i, f} = \setsuch{x \in X_i}{d(x, x^+_{i, f}) \leq \eps},\]
\[B^\eps_{i, f} = \setsuch{x \in X_i}{\delta(x, X^<_{i, f}) \geq \eps},\]
\[\overline{G}^\eps_f = \setsuch{g \in G}{\forall i = 1, \ldots, r,\quad \rho_i(g)(B^\eps_{i, f}) \subset b^\eps_{i, f} \text{ and } \restr{\rho_i(g)}{B^\eps_{i, f}} \text{ is $\eps$-Lipschitz}} \text{ and}\]
\[G^\eps_f = \bigcup_{\eps' < \eps} \overline{G}^{\eps'}_f.\]

From our point of view, the set~$G^\eps_f$ does not differ much from the set
\[G^{(\eps)}_f = \setsuch{g \in G}{g \text{ is } (\Pi, \eps)\text{-proximal and } d(f_g, f) \leq \eps}\]
that was introduced in~\ref{sec:3.6}. This is stated precisely in point~c) of the following lemma.
\begin{lemma*} ($k = \mathbb{R}$)
Let $f \in Z_\Pi$ and $\eps < \eps_f$.
\begin{enumerate}[label=\alph*)]
\item The semigroup~$G^\eps_f$ is open in~$G$. In particular it is Zariski-dense.
\item The semigroup~$G^\eps_f$ is strongly $2\eps$-Schottky on~$Y_\Pi$.
\item We may find $\eta = \eta_{f, \eps}$ such that $G^\eta_f \subset G^{(\eps)}_f$ and $G^{(\eta)}_f \subset G^\eps_f$.
\item Its limit cone is~$A^+$ and, when $\eps$ decreases to~$0$, the limit sets $\Lambda^\pm_{G^\eps_f}$ form a basis of neighborhoods of~$y^\pm_f$.
\end{enumerate}
\end{lemma*}
\begin{remark*}
It would have been just as natural to introduce the open semigroups~$G^{((\eps))}_f$ defined by
\[b^\eps_f = \setsuch{y \in Y_\Pi}{d(y, y^+_f) \leq \eps},\]
\[B^\eps_f = \setsuch{y \in Y_\Pi}{\delta(y, U^c_{y^-_f}) \geq \eps},\]
\[\overline{G}^{((\eps))}_f = \setsuch{g \in G}{g(B^\eps_f) \subset b^\eps_f \text{ and } \restr{g}{B^\eps_f} \text{ is $\eps$-Lipschitz}} \text{ and}\]
\[G^{((\eps))}_f = \bigcup_{\eps' < \eps} \overline{G}^{((\eps'))}_f.\]
In fact this semigroup does not differ much from~$G^\eps_f$: the reader may check that there exists $\zeta = \zeta_{f, \eps} > 0$ such that $G^\zeta_f \subset G^{((\eps))}_f$ and $G^{((\zeta))}_f \subset G^\eps_f$.
\end{remark*}

\begin{proof}~
\begin{enumerate}[label=\alph*)]
\item It is clear that $G^\eps_f$~is an open semigroup. It is nonempty because if $g$~is an element of~$G$ such that $f_g = f$, there exists $n \geq 1$ such that $g^n$ is in~$G^\eps_f$.
\item We proceed as in Lemma~6.2 of~\cite{Be}: Let $g$~be in~$G^\eps_f$. The restriction of~$\rho_i(g)$ to~$B^\eps_{i, f}$ is $\eps$-Lipschitz. Hence it has an attracting fixed point~$x^+_{i, g}$. Hence $\rho_i(g)$~is proximal. Let $X^<_{i, g} := X^<_{\rho_i(g)}$. Since $g(B^\eps_{i, f}) \subset b^\eps_{i, f}$, we have $d(x^+_{i, g},\; x^+_{i, f}) \leq \eps$ and $d(X^<_{i, g},\; X^<_{i, f}) \leq \eps$. Since $10\eps \leq \delta(x^+_{i, f},\; X^<_{i, f})$, we deduce that $\rho_i(g)$ is $2\eps$-proximal and that if $g'$~is another element of~$G^\eps_f$, we have $\delta(x^+_{i, g},\; X^<_{i, g'}) \geq 4\eps$. The semigroup~$G^\eps_f$ is indeed $(\Pi, 2\eps)$-Schottky on~$Y_\Pi$.
\item The passage from~$\eps$ to~$\eta$ and vice-versa comes from different natural distances that define the same topology on~$Z_\Pi$: we may, for example, choose $\eta < \frac{\eps}{2}$ such that for every projector~$\pi$ of~$V_i$ with image $x^+ \in \mathbb{P}(V_i)$ and with kernel $X^< \subset \mathbb{P}(V_i)$, we have the implications:
\[\|\pi - \pi_{i, f}\| < \eta \quad\implies\quad \left(d(x^+,\; x^+_{i, f}) < \frac{\eps}{2} \;\text{ and }\; \delta(X^<,\; X^<_{i, f}) < \frac{\eps}{2}\right);\]
\[\left(d(x^+,\; x^+_{i, f}) < \eta \;\text{ and }\; \delta(X^<,\; X^<_{i, f}) < \eta\right) \quad\implies\quad \|\pi - \pi_{i, f}\| < \frac{\eps}{2}.\]

The reasoning done in~b) and the choice of~$\eta$ prove that if $g$~is in $G^\eta_f$ then $g$ is $(\Pi, \eps)$-proximal and $d(f_g, f) := \sup_{1 \leq i \leq r} \|\pi_{i, g} - \pi_{i, f}\| \leq \eps$. Hence $g$~is in $G^{(\eps)}_f$.

Conversely, if $g$~is in~$G^{(\eta)}_f$, we have $d(f_g, f) \leq \eta$ and thus
\[d(x^+_{i, g},\; x^+_{i, f}) < \frac{\eps}{2} \;\;\text{ and }\;\; \delta(X^<_{i, g},\; X^<_{i, f}) < \frac{\eps}{2}.\]
Hence given that $g$ is $(\Pi, 2\eps)$-proximal, we have
\[\rho_i(g)(B^\eps_{i, f}) \subset \rho_i(g)(B^{\frac{\eps}{2}}_{i, f_g}) \subset b^{\frac{\eps}{2}}_{i, f_g} \subset b^\eps_{i, f}.\]
This is true for every~$i$, hence $g$~is in~$\overline{G}^\eps_f$. The same reasoning for some suitably chosen $\eps' < \eps$ shows that $g$~is in~$G^\eps_f$.
\item Note that if $g$~is an element such that $d(f_g, f) \leq \eps$, then there exists $n \geq 1$ such that $g^n \in G^{(\eps)}_f$. Our claims then follow from~c). \qedhere
\end{enumerate}
\end{proof}

Let $M_\eps$ be the compact subset of~$A^\bullet$ introduced in Lemma~\ref{sec:4.1}. For every closed convex cone~$\Omega$ in~$A^\times$ with nonempty interior, we call $\Omega_\eps$ the interior of the largest translate of~$\Omega$ such that $\Omega_\eps + 2M_{2\eps} \subset \Omega$ and we set
\[G^\eps_{f, \Omega} = \setsuch{g \in G^\eps_f}{\lambda(g) \in \Omega_\eps}.\]
\begin{lemma*} ($k = \mathbb{R}$)
The set~$G^\eps_{f, \Omega}$ is an open semigroup in~$G$ that is strongly $2\eps$-Schottky on~$Y_\Pi$. Its limit cone is~$\Omega$. When $\eps$ decreases to~$0$, the limit sets $\Lambda^\pm_{G^\eps_{f, \Omega}}$ form a basis of neighborhoods of~$y^\pm_f$.
\end{lemma*}
\begin{proof}
The only claim that does not directly follow from Lemma~\ref{sec:5.1} is the fact that $G^\eps_{f, \Omega}$ is a semigroup: let $g_1$ and~$g_2$ be two elements of~$G^\eps_{f, \Omega}$; since the semigroup $G^\eps_f$ is strongly $2\eps$-Schottky on~$Y_\Pi$, Lemma~\ref{sec:4.1} gives us
\[\lambda(g_1 g_2) \in \lambda(g_1) + \lambda(g_2) + 2M_{2\eps} \subset \Omega_\eps + \Omega \subset \Omega_\eps,\]
and so $g_1 g_2$ is also in $G^\eps_{f, \Omega}$.
\end{proof}

\subsection{The semigroups $G^\eps_f$ and $G^\eps_{f, \Omega}$ for non-Archimedean~$k$}
\label{sec:5.3}

In this subsection, we carry out the constructions analogous to those of~\ref{sec:5.2} when $k$ is non-Archimedean.

We will need the following elementary lemma.

\begin{lemma*} ($k$ non-Archimedean)
Let $V$~be a $k$-vector space endowed with an ultrametric norm, $\pi \in \End(V)$ a nonzero projector and $\eps < 1$. Then the open set
\[O^\eps_\pi := \setsuch{g \in \End(V)}{\|\sigma g \sigma' - \pi\| < \eps,\quad \forall (\sigma, \sigma') \in \{1, \pi\}^2}\]
is a semigroup.
\end{lemma*}
\begin{proof}
This is an application of the equality $\pi^2 = \pi$ and of the ultrametricity of the induced norm on~$\End(V)$: let $g, h$ be in~$O^\eps_\pi$; we then have, with $\sigma, \sigma'$ in~$\{1, \pi\}$,
\[\|\sigma g h \sigma' - \pi\| \;\leq\; \sup \left( \|(\sigma g - \pi) (h \sigma' - \pi)\|,\; \|\sigma g \pi - \pi\|,\; \|\pi h \sigma' - \pi\| \right) \;\leq\; \sup( \eps^2, \eps, \eps ) \;=\; \eps.\]
Hence $g h$ is in $O^\eps_\pi$.
\end{proof}

Let $\theta$ be a subset of~$\Pi$ and $f$ an element of~$Z_\theta$. We call~$L_f$ the stabilizer in~$G$ of~$f$: this is also the intersection of the two parabolic subgroups $y^+_f$ and $y^-_f$.

For $i = 1, \ldots, r$, we call $V^+_{i, f}$ the smallest nonzero vector subspace of~$V_i$ that is invariant under the action of the parabolic subgroup~$y^+_f$, $V^<_{i, f}$ the unique $L_f$-invariant subspace supplementary to~$V^+_{i, f}$ and $\pi_{i, f} \in \End(V_i)$ the projector onto~$V^+_{i, f}$ parallel to~$V^<_{i, f}$. We remark that $\pi_{i, f}$ has rank~$1$ if and only if $\alpha_i$ is in~$\theta$. We remind that $\mathrm{u}$ denotes a uniformizing element, we fix $\eps < 1$ and we set
\[G^\eps_f = \setsuch{g \in G}{\forall i = 1, \ldots, r,\; \exists p_i(g) \in \mathbb{Z},\quad \mathrm{u}^{-p_i(g)}\rho_i(g) \in O^\eps_{\pi_{i, f}}}.\]

\begin{lemma*}($k$ non-Archimedean)
Let $f \in Z_\theta$ and $\eps < 1$.
\begin{enumerate}[label=\alph*)]
\item The semigroup $G^\eps_f$ is open and closed in~$G$ and is of type~$\theta$. In particular, it is Zariski-dense.
\item For $\eps$ sufficiently small, we can find $\eta = \eta_{f, \eps}$ such that $G^\eta_f$ is strongly $(\theta, \eps)$-Schottky.
\item The limit cone of~$G^\eps_f$ is~$A^+_\theta$. When $\eps$ decreases to~$0$, the limit sets of~$G^\eps_f$ form a basis of neighborhoods of the points~$y^\pm_f$.
\item For every $g, g'$ in~$G^\eps_f$, we have $\lambda(g g') = \lambda(g)\lambda(g')$.
\end{enumerate}
\end{lemma*}

\begin{proof}~
\begin{enumerate}[label=\alph*)]
\item It is clear that $G^\eps_f$ is open, closed and nonempty. It is of type~$\theta$ thanks to Lemma~\ref{sec:3.2}.ii).
\item This is the same proof as for $k = \mathbb{R}$.
\item Ditto.
\item This follows from the equalities:
\[\chi_i(\lambda(g)) = \mathrm{u}^{p_i(g)} \quad\text{and}\]
\[p_i(g g') = p_i(g) + p_i(g'). \qedhere\]
\end{enumerate}
\end{proof}

\begin{corollary*} ($k$ non-Archimedean)
Let $\Omega$ be a closed convex cone with rational support (\ie such that $\Omega$ and $\Omega \cap A^+$ generate the same vector subspace~$B$ of~$A^\bullet$). Let $\theta$~be the smallest subset of~$\Pi$ such that $A^\times_\theta$ contains~$\Omega$, $f$~be an element of~$Z_\theta$ and $\eps \leq 1$.

We set
\[G^\eps_{f, \Omega} = \setsuch{g \in G^\eps_f}{\lambda(g) \in \Omega}.\]

Then the sets $G^\eps_{f, \Omega}$ are open subsemigroups of type~$\theta$ whose Lyapunov cone is~$\Omega$.
\end{corollary*}
\begin{proof}
This follows from the fact that the map $\lambda: G^\eps_f \to A^+$ is a locally constant semigroup morphism.
\end{proof}

\section{The set of limit directions $\mathcal{L}_\Gamma$}
\label{sec:6}

In this section, we focus first of all on open subsemigroups of~$G$. We show that for $g_1, g_2$ in~$G$, we can find two disjoint open subsemigroups~$H_1, H_2$ containing respectively the points $g_1$ and~$g_2$ if and only if the hyperbolic components of $g_1$ and~$g_2$ do not lie in the same ``one-parameter subsemigroup'' (\ref{sec:6.2}).

We suppose starting from \ref{sec:6.3} that $k = \mathbb{R}$. We then show that the intersection of the Zariski-dense subsemigroup~$\Gamma$ with an open subsemigroup is still Zariski-dense as long as it is nonempty.

These two properties are what motivates the definition of the set~$\mathcal{L}_\Gamma$ of the limit directions of~$\Gamma$: it is the set of the hyperbolic elements~$g$ of~$G$ such that every open subsemigroup of~$G$ that contains~$g$ intersects~$\Gamma$. We also give other equivalent definitions of~$\mathcal{L}_\Gamma$ that justify the terminology of the ``limit directions'' (Theorem~\ref{sec:6.4}.b).

We show in~\ref{sec:6.4} that the set~$\mathcal{L}_\Gamma$ meets the interior of a Weyl chamber~$f$ if and only if $f$~is quasiperiodic and that in this case, the intersection $f \cap \mathcal{L}_\Gamma$ can be naturally identified to the limit cone~$L_\Gamma$.

\subsection{Open semigroups and hyperbolic elements}
\label{sec:6.1}

The following proposition will be useful to us.

\begin{proposition*}
Let $g$~be an element of~$G$. Then
\begin{enumerate}[label=\alph*)]
\item Every open subsemigroup~$H$ of~$G$ that contains~$g$ also contains~$g^n_h$ for some integer $n \geq 1$.
\item Every open subsemigroup~$H$ of~$G$ that contains~$g_h$ also contains~$g^n$ for some integer $n \geq 1$.
\end{enumerate}
\end{proposition*}

\begin{remark*}
Recall that when $k$ is non-Archimedean, the hyperbolic component~$g_h$ might not exist. However $g_h^n := (g^n)_h$ makes sense as soon as $n$~is a multiple of an integer $n_G$.
\end{remark*}

Let us start with a preparatory lemma.

\begin{lemma*}
Same notation. We assume that $\lambda(g) = 0$.
\begin{enumerate}[label=\alph*)]
\item Every open subsemigroup~$H$ of~$G$ that contains~$g$ also contains the identity element~$e$.
\item For every neighborhood~$B$ of~$e$ in~$G$, there exists an integer $n \geq 1$ such that $g^n \in B^n$. Moreover, if $k = \mathbb{R}$ or~$\mathbb{C}$, there exists an integer $n_o \geq 1$ such that for every $n \geq n_o$, we have $g^n \in B^n$.
\end{enumerate}
\end{lemma*}

\begin{proof}~
\begin{enumerate}[label=\alph*)]
\item The Jordan decomposition of~$g$ can be written $g = g_e g_u$. Let $B$~be a neighborhood of~$e$ such that $B g \subset H$. There exists a sequence $n_p$ such that $\lim_{p \to \infty} g^{n_p}_e = e$. Hence, for $p \gg 0$, we have
\[g^{n_p}_u \in B g^{n_p} \subset H.\]
Hence $H$~contains a unipotent element. But every neighborhood of a unipotent element contains an elliptic element (if $k$~is non-Archimedean, this results from the density of the semisimple elements; if $k = \mathbb{R}$ or~$\mathbb{C}$, this is true for $G = \SL(2, k)$ and the general case follows by the Jacobson-Morozov theorem). We deduce that $H$~contains an elliptic element and then, following the previous reasoning, that it contains~$e$.

\item It suffices to prove this statement separately for elliptic~$g$ and for unipotent~$g$.

When $g$~is elliptic, or when $k$~is non-Archimedean, this follows from the fact that the group generated by~$g$ is bounded.

When $g$~is unipotent and $k = \mathbb{R}$ or~$\mathbb{C}$, using the Jacobson-Morozov theorem, we reduce the problem to the case $G = \SL(2, \mathbb{R})$. In a suitable basis, we then have $g = \left( \begin{smallmatrix} 1 & 1 \\ 0 & 1 \end{smallmatrix} \right)$. Let $B'$~be a neighborhood of~$e$ such that $(B')^3 \subset B$. We choose $a > 1$ such that the element $h := \left( \begin{smallmatrix} a & 0 \\ 0 & a^{-1} \end{smallmatrix} \right)$ and its inverse~$h^{-1}$ are in~$B'$. We can then find $n_o \geq 1$ such that for every $n \geq n_o$, the element $u_n := \left( \begin{smallmatrix} 1 & na^{-2n} \\ 0 & 1 \end{smallmatrix} \right)$ is in~$B'$. The equality $g^n = h^n u_n h^{-n}$ then proves that $g$~is in $(B')^{2n+1} \subset B^n$. \qedhere
\end{enumerate}
\end{proof}

\begin{proof}[Proof of the proposition.]
We may assume that $g_h$~exists. Let $L$~be the centralizer of~$g_h$.
\begin{enumerate}[label=\alph*)]
\item There exists a neighborhood~$L_\eps$ of~$e$ in~$L$ such that $g L_\eps \subset H$. Thanks to Lemma~a), we can find $n \geq 1$ such that $(g_e g_u L_\eps)^n$ contains~$e$. But then we have
\[g_h^n \in g_h^n (g_e g_u L_\eps)^n = (g L_\eps)^n \subset H.\]
\item There exists a neighborhood~$L_\eps$ of~$e$ in~$L$ such that $g_h L_\eps \subset H$. Thanks to Lemma~b), we can find $n \geq 1$ such that $(g_e g_u)^n \in (L_\eps)^n$. But then we have
\[g^n \in g_h^n (L_\eps)^n = (g_h L_\eps)^n \subset H. \qedhere\]
\end{enumerate}
\end{proof}

\subsection{The open semigroups $G^\eps_g$}
\label{sec:6.2}

We fix a faithful representation~$\rho$ of~$G$ and we call, for $g$ in~$G$,
\[B(g, \eps) = \setsuch{g \in G}{\|\rho(g') - \rho(g)\| \leq \eps} \quad\text{and}\]
\[G^\eps_g = \bigcup_{n \geq 1} B(g, \eps)^n\]
the open semigroup generated by this small neighborhood of~$g$. We call $Y^\pm_g$ the inverse image of the point~$y^\pm_g$ by the natural projection $Y_\Pi \to Y_{\theta_g}$.

\begin{proposition*}
Let $g$~be an element of~$G$ such that $\lambda(g) \neq 0$.
\begin{enumerate}[label=\alph*)]
\item ($k = \mathbb{R}$) When $\eps$ decreases to~$0$, the limit cones~$L_{G^\eps_g}$ form a basis of conical neighborhoods of the half-line~$L_g$ in~$A^\times$ and the limit sets~$\Lambda^\pm_{G^\eps_g}$ form a basis of neighborhoods of the subset~$Y^\pm_g$ in~$Y_\Pi$.
\item ($k$ non-Archimedean) There exists $\eps_0 > 0$ such that for every $\eps < \eps_0$, the limit cone~$L_{G^\eps_g}$ is equal to~$L_g$. In particular, $G^\eps_g$ is of type $\theta := \theta_g$. Moreover, when $\eps$ decreases to~$0$, the limit sets~$\Lambda^\pm_{G^\eps_g}$ form a basis of neighborhoods of the point~$y^\pm_g$ in~$Y^\pm_\theta$.
\end{enumerate}
\end{proposition*}

\begin{proof}
Let us first prove the claim about the limit cone. Let $\eta > 0$ if $k = \mathbb{R}$ or~$\mathbb{C}$ and $\eta = 0$ if $k$~is non-Archimedean. It suffices to find a neighborhood $B(g, \eps)$ such that for every $h$ in~$B(g, \eps)^n$ with $n \geq 1$ and for every $i = 1, \ldots, r$, we have
\[-\eta \;\leq\; \frac{1}{n} \log\left( |\chi_i|(\lambda(h)) \right) - \log\left( |\chi_i|(\lambda(g)) \right) \;\leq\; \eta.\]
Since we have the equality $\chi_i(\lambda(h)) = \lambda_1(\rho_i(h))$, this follows from the lemma below.

Let us now prove the claim about the limit set~$\Lambda^+$ (we proceed in the same fashion for~$\Lambda^-$). First notice that, when $k = \mathbb{R}$, since $G^\eps_g$ contains a power of~$g_h$, the limit set $\Lambda^+_{G^\eps_g}$ contains~$Y^+_g$. To conclude, it then suffices to show that for $\alpha_i$ in~$\theta_g$, the limit sets of~$G^\eps_g$ in~$\mathbb{P}(V_i)$ form, when $\eps$ decreases to~$0$, a basis of neighborhoods of the point~$x^+_{i, f_g}$. This fact follows from the proximality of~$g$ in~$\mathbb{P}(V_i)$.
\end{proof}

We used the following lemma:
\begin{lemma*}
Let $V$~be a finite-dimensional $k$-vector space and $g \in \End(V)$. Let $\eta > 0$ if $k$~is Archimedean and $\eta = 0$ otherwise.

Then there exists a neighborhood~$B_g$ of~$g$ in~$\End(V)$ such that for every $h$ in~$(B_g)^n$ with $n \geq 1$ and for every $j = 1, \ldots, \dim(V)$, we have
\[-\eta \;\leq\; \frac{1}{n}\log(\lambda_j(h)) - \log(\lambda_j(g)) \;\leq\; \eta.\]
\end{lemma*}
\begin{proof}
Given that $\lambda_1(\Lambda^j h) = \lambda_1(h) \cdots \lambda_j(h)$, it suffices to prove these inequalities for $j = 1$.

Notice that we can find a norm on~$V$ such that $\|g\| \leq \lambda_1(g) e^{\frac{\eta}{2}}$. We can then choose~$B_g$ such that for every $h'$ in~$B_g$, we have $\|h'\| \leq \lambda_1(g)e^\eps$. So if $h = h_1 \cdots h_n$ with $h_i \in B_g$, we have
\[\frac{1}{n}\log(\lambda_1(h)) \;\leq\; \frac{1}{n}\log\|h\| \;\leq\; \frac{1}{n} \sum_{1 \leq p \leq n} \log \|h_p\| \;\leq\; \log(\lambda_1(g)) + \eps.\]

It remains to obtain the lower bound for~$\lambda_1(h)$. Let $s$~be the largest integer such that $\lambda_1(g) = \cdots = \lambda_s(g)$. Note that $\lambda_1(\Lambda^s h) \leq \lambda_1(h)^s$ and that $\Lambda^s g$ is proximal in~$\mathbb{P}(\Lambda^s V)$. We may thus assume that $g$~is proximal.

We call $V^+$ the eigenline of~$g$ corresponding to the eigenvalue with the largest absolute value and $V^<$ its unique $g$-invariant supplementary hyperplane. We can find a norm on~$V$ such that $\|\restr{g}{V^<}\| < \lambda_1(g)$ and such that for every $v^+ \in V^+$ and $v^< \in V^<$, we have $\|v^+ + v^<\| \leq \sup(\|v^+\|, \|v^<\|)$. Let
\[C := \setsuch{v = v^+ + v^<}{v^+ \in V^+,\; v^< \in V^< \text{ and } \|v^+\| \geq \|v^<\|}.\]
We can choose the neighborhood~$B_g$ in such a way that for every $h'$ in~$B_g$, we have
\[h'(C) \subset C \quad\text{and}\quad \|h'(v)\| \geq \lambda_1(g) e^{-\eps}\|v\|\quad \forall v \in C.\]
We then calculate, for $h$ in~$(B_g)^n$ and $v$ in~$C$
\[\|h(v)\| \geq \lambda_1(g)^n e^{-n \eps} \|v\|.\]
Hence
\[\frac{1}{n} \log \|h\| \geq \log \lambda_1(g) - \eps.\]
The equality $\log(\lambda_1(h)) = \lim_{p \to \infty} \frac{1}{p} \log \|h^p\|$ then allows us to get the desired lower bound for~$\lambda_1(h)$.
\end{proof}

\begin{definition*}
We define a \emph{hyperbolic one-parameter subsemigroup} of~$G$ to be a semigroup~$\mathcal{L}$ such that
\begin{itemize}
\item if $k = \mathbb{R}$ or~$\mathbb{C}$, $\mathcal{L} = \gamma([0, \infty))$ where $\gamma: \mathbb{R} \to G$ is a one-parameter subgroup composed of hyperbolic elements.
\item if $k$ is non-Archimedean, $\mathcal{L}$~is isomorphic to~$\mathbb{N}$, is composed of hyperbolic elements and is maximal for these properties.
\end{itemize}
\end{definition*}

Every hyperbolic element $h \neq e$ lies in a unique hyperbolic subsemigroup, that we denote by~$\mathcal{L}_h$. If $h = e$, we set $\mathcal{L}_h := \{e\}$. For every element $g$ of~$G$, we set $\mathcal{L}_g := \mathcal{L}_h$ where $h$~is the hyperbolic component of~$g$ or of some power of~$g$. Note that $\mathcal{L}_g$ can also be defined by the equality $\lambda(\mathcal{L}_g) = L_g \cap A^+$ and the inclusion $\mathcal{L}_g \subset f_g$.

\begin{corollary*}
Let $g_1$, $g_2$ be two elements of~$G$. We then have the equivalence: $\mathcal{L}_{g_1} \neq \mathcal{L}_{g_2} \iff$ There exists two open subsemigroups $H_1$, $H_2$ of~$G$ such that $g_1 \in H_1$, $g_2 \in H_2$ and $H_1 \cap H_2 = \emptyset$.
\end{corollary*}
\begin{proof}
($\Rightarrow$) By assumption, we have $(y^+_{g_1}, y^-_{g_1}, L_{g_1}) \neq (y^+_{g_2}, y^-_{g_2}, L_{g_2})$. Hence it suffices to take $H_1 = G^\eps_{g_1}$, $H_2 = G^\eps_{g_2}$ with $\eps$ sufficiently small and to apply Proposition~\ref{sec:6.2}.

($\Leftarrow$) This follows from Proposition~\ref{sec:6.1}.
\end{proof}

\subsection{Open semigroups and Zariski-dense semigroups for $k = \mathbb{R}$}
\label{sec:6.3}

The following propositions say that open semigroups can be used as a ``filter'' to analyse the Zariski-dense semigroups. We will use it in~\ref{sec:7.4}.

\begin{proposition*} ($k = \mathbb{R}$)
Let $\Gamma$~be a Zariski-dense subsemigroup of~$G$ and $H$~be an open subsemigroup of~$G$. If $\Gamma \cap H$ is nonempty, then $\Gamma \cap H$~is still Zariski-dense.
\end{proposition*}
\begin{proof}
Let $g$~be an element of $\Gamma \cap H$ and $g = g_e g_h g_u$ be its Jordan decomposition. We lose no generality in assuming that $g_h$ is in~$A^+$. Let $\theta := \setsuch{\alpha \in \Pi}{\alpha(g_h) = 1}$. We have $L_\theta := \setsuch{g' \in G}{g' g_h = g_h g'}$, $U_\theta := \setsuch{g' \in G}{\lim_{n \to -\infty} g^n g' g^{-n} = e}$ and $U^-_\theta := \setsuch{g' \in G}{\lim_{n \to +\infty} g^n g' g^{-n} = e}$. Multiplication induces a diffeomorphism from the product $U^-_\theta \times L_\theta \times U_\theta$ to a Zariski-open subset $\Omega_\theta$ of~$G$. Let $U^-_\eps$, $L_\eps$, $U_\eps$ be three neighborhoods of~$e$ in $U^-_\theta$, $L_\theta$ and~$U_\theta$ such that
\[U^-_\eps g L_\eps U_\eps \subset H.\]
For $n \geq 1$, we have
\[U^-_\eps g^n_h(g_e g_u L_\eps)^n U_\eps \subset H.\]
By Lemma~\ref{sec:6.1}.a, we can find $n \geq 1$ such that $(g_e g_u L_\eps)^n$ contains~$e$. Replacing if necessary $g$ by~$g^n$ and consequently reducing our neighborhoods, we may assume that
\[U^-_\eps g_h L_\eps U_\eps \subset H.\]

Let us then show that for every $\gamma$ in~$\Gamma \cap \Omega_\theta$, there exists $n_o \geq 1$ such that for every $n \geq n_o$, we have $g^n \gamma g^{-n} \in \Gamma \cap H$. This will prove that $\Gamma \cap H$ is Zariski-dense by the same argument as at the end of~\ref{sec:3.6}. We write $\gamma = v l u$ with $v \in U^-_\theta$, $l \in L_\theta$ and $u \in U_\theta$. We have
\[g^n \gamma g^n = v_n g_h^{2n} l_n u_n\]
with
\[v_n = g^n v g^{-n},\quad l_n = g_e^n g_u^n l g_e^n g_u^n \quad\text{and}\quad u_n = g^{-n} u g^n.\]
For large enough~$n$, we have, by Lemma~\ref{sec:6.1}.b,
\[v_n \in U^-_\eps,\quad l_n \in L^n_\eps \quad\text{and}\quad u_n \in U_\eps.\]
Hence
\[g^n \gamma g^n \in U^-_\eps g_h^{2n} L_\eps^n U_\eps \subset (U^-_\eps g_h L_\eps U_\eps)^n \subset H,\]
which finishes the proof.
\end{proof}

\subsection{The set of limit directions for $k = \mathbb{R}$}
\label{sec:6.4}

In this subsection, we still assume that $k = \mathbb{R}$, so that $\theta_\Gamma = \Pi$.

\begin{definition*}
We denote by~$\mathcal{L}_\Gamma$ the set of the hyperbolic elements $g$ of~$G$ such that every open subsemigroup~$H$ of~$G$ containing~$g$ intersects~$\Gamma$.

We set $\mathcal{L}'_\Gamma := \setsuch{g \in \mathcal{L}_\Gamma}{g \text{ is $\mathbb{R}$-regular}}$ and $L'_\Gamma := L_\Gamma \cap A^{\times \times}$.
\end{definition*}

\begin{remark*}
In this definition, we limit ourselves to hyperbolic elements because of Proposition~\ref{sec:6.1}.
\end{remark*}

We endow $G$ with a left-invariant Riemannian metric and we call~$d$ the corresponding distance.

\begin{theorem*} ($k = \mathbb{R}$; $\mathbf{G}$~is a connected reductive $\mathbb{R}$-group, $G = \mathbf{G}_\mathbb{R}$ and $\Gamma$~is a Zariski-dense subsemigroup of~$G$.)
\begin{enumerate}[label=\alph*)]
\item $\mathcal{L}_\Gamma$ is the closure of the union of the semigroups~$\mathcal{L}_g$ for~$g \in \Gamma$ (resp.\ for $\mathbb{R}$-regular $g$ in~$\Gamma$).
\item $\mathcal{L}_\Gamma$ is the closure of the set $\mathcal{L}^o_\Gamma$ of the hyperbolic elements~$g$ of~$G$ for which there exist sequences $n_p \in \mathbb{N}$ and $\gamma_p \in \Gamma$ such that $\lim_{p \to \infty} d(g^{n_p}, \gamma_p) = 0$.
\item $\mathcal{L}'_\Gamma$ is dense in~$\mathcal{L}_\Gamma$ and $L'_\Gamma$~is dense in~$L_\Gamma$.
\item A chamber $f \in Z_\Pi$ is quasiperiodic if and only if $\mathcal{L}_\Gamma$ intersects the interior of the chamber~$f$.
\item For every quasiperiodic chamber $f$ of~$\Gamma$, we have $\lambda(f \cap \mathcal{L}_\Gamma) = L_\Gamma$.
\end{enumerate}
\end{theorem*}

\begin{remarks*}~
\begin{itemize}
\item Parts a) and~b) furnish alternative definitions for the set of limit directions of~$\Gamma$.
\item Parts c), d) and~e) explain how to calculate $\mathcal{L}_\Gamma$ from $\Lambda_\Gamma$, $\Lambda^-_\Gamma$ and $L_\Gamma$ and vice-versa (we remind that, by \ref{sec:3.6}.iv, the set~$F_\Gamma$ of the quasiperiodic chambers of~$\Gamma$ can be identified to the set of elements of $\Lambda_\Gamma \times \Lambda^-_\Gamma$ that are in general position). Indeed, they state that the set~$\mathcal{L}_\Gamma$ is the closure of the set
\[\mathcal{L}'_\Gamma = \bigcup_{f \in F_\Gamma} \setsuch{g \in f}{\lambda(g) \in L'_\Gamma}.\]
\item We leave out the analogous statement for a local field: in this case the set~$\mathcal{L}_\Gamma$ lives in a space which is related to the set of the hyperbolic elements of~$G$ in the same way as $A^\times$ is related to~$A^+$...
\end{itemize}
\end{remarks*}

\begin{proof}~
\begin{enumerate}[label=\alph*)]
\item By construction, the subset $\mathcal{L}_\Gamma$ is closed. It contains the semigroups~$\mathcal{L}_g$ for $g$ in~$\Gamma$ by \ref{sec:6.1}. Conversely, let $g$ be in~$\mathcal{L}_\Gamma$; we then apply Proposition~\ref{sec:6.3} with $\eps = \frac{1}{n}$: the intersection $G^{\frac{1}{n}}_g \cap \Gamma$ is Zariski-dense and thus contains an $\mathbb{R}$-regular element $g_n$. Hence the sequence $\mathcal{L}_{g_n}$ converges to~$\mathcal{L}_g$ by~\ref{sec:6.2} (\ie there exists a sequence $t_n > 0$ such that $\lim_{n \to \infty} g_n^{t_n} = g$).
\item It is clear that if $g$~is in~$\mathcal{L}^o_\Gamma$, then every open subsemigroup of~$G$ containing~$g$ intersects~$\Gamma$. Hence $\overline{\mathcal{L}^o_\Gamma} \subset \mathcal{L}_\Gamma$. Conversely, if $g$~is an $\mathbb{R}$-regular element of~$\Gamma$, then for every $t > 0$, the element $g^t$ is in~$\mathcal{L}^o_\Gamma$: it suffices to choose the sequence~$n_p$ so that the sequence~$n_p t$ converges to~$0$ in~$\mathbb{R}/\mathbb{Z}$. Hence $\mathcal{L}_\Gamma \subset \overline{\mathcal{L}^o_\Gamma}$.
\item The first statement follows from~a), the second one follows from the convexity of~$L_\Gamma$ and from the fact that $L'_\Gamma$ is nonempty.
\item and e) First of all notice that, thanks to~a), we have $\lambda(\mathcal{L}_\Gamma) \subset L_\Gamma$.

Let $f$~be a chamber whose interior intersects~$\mathcal{L}_\Gamma$. We call~$g$ some element of~$\mathcal{L}_\Gamma$ that is in the interior of~$f$. The reasoning from a) gives us a sequence~$g_n$ of $\mathbb{R}$-regular elements of~$\Gamma$ such that $f_{g_n}$ converges to~$f$. Hence $f$~is quasiperiodic.

Conversely, let $f$~be a quasiperiodic chamber for~$\Gamma$. Let $L$~be a half-line in~$L_\Gamma$. Let $\mathcal{L}$ denote the one-parameter subsemigroup of~$f$ such that $\lambda(\mathcal{L}) = L$. By~\ref{sec:4.2}, there exists a sequence of $\mathbb{R}$-regular elements~$g_n$ of~$\Gamma$ such that $\mathcal{L}_{g_n}$ converges to~$\mathcal{L}$. Then every open subsemigroup of~$G$ that intersects~$\mathcal{L}$ contains some power of one of the elements~$g_n$, by~\ref{sec:6.1}. Hence $\mathcal{L}$~is contained in~$\mathcal{L}_\Gamma$. This proves that $\lambda(f \cap \mathcal{L}_\Gamma) = L_\Gamma$. We deduce, given that $L'_\Gamma$~is nonempty, that $f \cap \mathcal{L}_\Gamma$ intersects the interior of the facet~$f$. \qedhere
\end{enumerate}
\end{proof}

\section{The limit cone~$L_\Gamma$ when $k = \mathbb{R}$}
\label{sec:7}

In this section, we suppose that $k = \mathbb{R}$. In this case, the logarithm map indentifies the $\mathbb{R}$-vector spaces $A^\bullet$ and~$\mathfrak{a}$, as well as the cones $L_\Gamma$ and~$\ell_\Gamma$. The goal of this section is to prove that in this case, the limit cone~$L_\Gamma$ has nonempty interior.

Recall that this statement is false over a non-Archimedean field (cf.~\ref{sec:5.3}).

The idea of the proof is to reduce the problem to the case of a subsemigroup of~$G^\eps_g$ with generators $(\gamma_j)_{1 \leq j \leq s}$ (cf.~\ref{sec:7.4}). We then consider one-parameter semigroups $\setsuch{\gamma_j(t)}{t \geq 1}$ lying in~$G^\eps_g$ and such that $\gamma_j(1) = \gamma_j$. We show on the one hand, by using Hardy fields (cf.~\ref{sec:7.2}), that the Lyapunov cone~$\ell_\Delta$ of the semigroup~$\Delta$ generated by all of these semigroups is contained in the vector space spanned by~$\ell_\Gamma$ (cf.~\ref{sec:7.4}). On the other hand, an elementary argument proves that such a semigroup~$\Delta$ has nonempty interior (cf.~\ref{sec:7.1}).

\subsection{Semigroups with nonempty interior}
\label{sec:7.1}

\begin{lemma*} ($k = \mathbb{R}$)
Let $t \mapsto \gamma_j(t)$, for $j = 1, \ldots, s$, be one-parameter subgroups of~$G$ and let $H$~be the semigroup generated by $(\gamma_j([1, \infty)))_{1 \leq j \leq s}$. We suppose that $H$~is Zariski-dense in~$G$; then $H$~has nonempty interior.
\end{lemma*}

\begin{remark*}
One can also prove the following analagous result, which will not be useful to us: \emph{In a simple real Lie group, every non-discrete Zariski-dense closed subsemigroup has nonempty interior.}
\end{remark*}

\begin{proof}
For $h$ in~$H$, we denote by~$C_h$ the cone in the Lie algebra~$\mathfrak{g}$ of~$G$ given by:
\[C_h := \setsuch{X \in \mathfrak{g}}{\exists \text{ a } C^1 \text{ path } \gamma: [0, 1) \to G, \quad
\begin{cases}
\gamma_0 = e \\
\frac{d}{dt} \restr{\gamma_t}{t = 0} = X \\
\forall t \in [0, 1],\; \gamma_t h \in H
\end{cases}}.\]

Clearly $C_h$~is a cone. On the other hand, we have
\[C_{h h'} \supset C_h + \Ad_h(C_{h'}) \quad \forall h, h' \in H \quad (*).\]

It suffices for this to consider the path in~$H$: $\gamma_t h \gamma'_t h' = \gamma_t (h \gamma'_t h^{-1}) h h'$.

Let $W$~be the vector space spanned by the cones~$C_h$. Let us show that $W = \mathfrak{g}$. It follows from~(*) that the space~$W$ is $\Ad_H$-invariant. Since $\Ad_H$ is Zariski-dense in~$Ad_G$, $W$~is an ideal of~$\mathfrak{g}$. Moreover, the generators~$X_j$ of the one-parameter groups $\gamma_j(t)$ are in~$W$. Hence $H$~is contained in the connected Lie group~$I$ with Lie algebra~$W$. Hence $I$~is Zariski-dense in~$G$ and $W = \mathfrak{g}$.

Let us now choose elements $h_i$ in~$H$ and $X_i$ in~$C_{h_i}$, for $i = 1, \ldots, n$, with $n$ the largest possible integer such that, setting $k_0 = 1,\; \ldots,\; k_i = h_1 \cdots h_i,\; \ldots$, the family $(Y_i := \Ad_{k_{i-1}}(X_i))_{1 \leq i \leq n}$ is linearly independent. This family $(Y_i)_{1 \leq i \leq n}$ is then a basis of~$\mathfrak{g}$. Indeed, otherwise we could find $h_{n+1}$ in~$H$ and $X_{n+1}$ in~$C_h$ such that $\Ad_{k_n^{-1}}(X_{n+1})$ is not in the span of $Y_1, \ldots, Y_n$, which would contradict the maximality of~$n$.

We denote by $t \mapsto \gamma_{i, t}$ some paths tangent to~$X_i$ at $t = 0$ such that $\gamma_{i, t} h_i$~is in~$H$, and we set $\gamma'_{i, t} := k_{i-1} \gamma_{i, t} k^{-1}_{i-1}$. The map
\[\fundef{\psi:}{[0,1]^n}{G}
{(t_1, \ldots, t_n)}{\gamma'_{1, t} \cdots \gamma'_{n, t} h_1 \cdots h_n = \gamma_{1, t} h_1 \cdots \gamma_{n, t} h_n}\]
has its image in~$H$. Its differential at~$0$ is bijective since $(Y_i)_{1 \leq i \leq n}$ is a basis of~$\mathfrak{g}$. The local inversion theorem then proves that $\psi((0, 1)^n)$ contains an open set. Hence $H$~has nonempty interior.
\end{proof}

\subsection{Hardy fields}
\label{sec:7.2}

Let us recall the main properties of the Hardy fields that we shall need. An excellent reference is (\cite{Ro} p.297--299).

Let $\mathcal{A}$~be the ring of germs at~$+\infty$ of real-valued $C^\infty$ functions on~$\mathbb{R}$. In other terms, $\mathcal{A}$~is the set of real-valued smooth functions defined on a half-line~$[a, \infty)$, modulo the equivalence relation that identifies two functions that coincide on a half-line~$[b, \infty)$.

\begin{definition*}
A \emph{Hardy field} is a subfield of the ring~$\mathcal{A}$ that is invariant by derivation.
\end{definition*}

Hardy fields are interesting because a nonzero function~$y$ that belongs to a Hardy field can vanish only finitely many times (because the function $\frac{1}{y}$ has to be smooth in a neighborhood of~$+\infty$). Hardy fields are ordered real fields that have been introduced with the goal of studying asymptotic developments (cf.~\cite{Bou}). The field~$\mathbb{R}$ of constant functions and the field $\mathbb{R}(x)$ of rational functions are Hardy fields. There are many others, since every solution of a polynomial equation (resp.\ of a first-order polynomial differential equation) with coefficients in a Hardy field is still in a Hardy field. More precisely:

\begin{proposition*}[\cite{Ro}, Theorems 1 and~2]
Let $K$~be a Hardy field, $P \in K[Y]$ be a polynomial with coefficients in~$K$ and $y$~be an element of~$\mathcal{A}$ such that $P(y) = 0$ (resp.\ $\frac{dy}{dx} = P(y)$). Then there exists a Hardy field~$K'$ containing $K$ and~$y$.
\end{proposition*}

In particular, if $K$~is a Hardy field and $f$~is a nonzero element of~$K$, then there exists a Hardy field~$K'$ containing $K$, $e^{|f|}$, $\log |f|$ and $|f|^\alpha$ for every real~$\alpha$.

\subsection{Maps with finite fibers}
\label{sec:7.3}

The following elementary lemma will be useful for us.

\begin{lemma*}
Let $M$~be a real analytic manifold, $\psi: M \times \mathbb{R} \to \mathbb{R}$ be an analytic function, $Z = \psi^{-1}(0)$ and $p: Z \to M$ be the restriction to~$Z$ of the first projection. Suppose that, for every $m$ in~$M$, $p^{-1}(m)$ is finite. Then there exists a nonempty open subset~$U$ of~$M$ such that $p^{-1}(\overline{U})$ is compact.
\end{lemma*}

\begin{proof}
We may assume that $Z$ and~$M$ are smooth. Let $Z' := \setsuch{z \in Z}{dp(z) \text{ is surjective}}$ and $Z'' := Z \setminus Z'$. By the analytic Sard theorem, $p(Z'')$ has codimension at least~$1$ in~$M$. Reducing~$M$ if necessary, we may assume that $Z''$~is empty, \ie that $p$~is a local diffeomorphism.

By contradiction: suppose that the conclusion does not hold. We then conscruct by induction an infinite sequence $(C_n)_{n \geq 1}$ of disjoint compact subsets of~$Z$ with nonempty interior such that the restriction $\restr{p}{C_n}$ is injective and such that the sequence of the image compact sets $K_n := p(C_n)$ is decreasing; this is possible since, given that $p^{-1}(K_n)$ is not compact, we can choose for $C_{n+1}$ a small neighborhood of a point in the interior of~$p^{-1}(K_n)$ which is not in $C_1 \cup \cdots \cup C_n$. But then if $m$~is a point in the intersection of the compact sets~$K_n$, the fiber~$p^{-1}(m)$ is infinite, which is a contradiction.
\end{proof}

\subsection{The cone $\ell_\Gamma$ has nonempty interior}
\label{sec:7.4}

We can now prove this statement.

Thanks to Proposition~\ref{sec:6.3}, we may assume that $\Gamma$ is contained in an open subsemigroup $G^\eps_g$ whose elements are all $\mathbb{R}$-regular (cf.~\ref{sec:6.2}). We may also assume that $\Gamma$~is generated by a finite family of elements~$(\gamma_j)_{1 \leq j \leq s}$ such that each of the groups generated by~$\gamma_j$ is Zariski-connected. Note that $\gamma_j$~is semisimple and let $\gamma_j = m_j a_j = a_j m_j$ be the Jordan decomposition of~$\gamma_j$, with $m_j$~elliptic and $a_j$~hyperbolic. Let us call $M_j$ the closure in~$G$ of the group generated by~$m_j$: this is a compact group. Replacing if necessary $\gamma_j$ by some power, we may assume that for every $t \geq 1$, we have $M_j a_j^t \subset G^\eps_g$.

Since $\ell_\Gamma$~is convex, to show that $\ell_\Gamma$ has nonempty interior, it suffices to show that every linear form on the $\mathbb{R}$-vector space~$\mathfrak{a}$ that vanishes on~$\ell_\Gamma$ is identically zero. So let $\beta_1, \ldots, \beta_r$ be real numbers such that for every $\gamma$ in~$\Gamma$, we have $\prod_{i=1}^r \lambda_1(\rho_i(\gamma))^{\beta_i} = 1$. Our goal is to show that $\beta_1 = \cdots = \beta_r = 0$.

Let $\Delta$~be the subsemigroup of~$G^\eps_g$ generated by all the subsets $M_j a_j^t$ with $j = 1, \ldots, s$ and $t \geq 1$. This semigroup $\Delta$ contains $\Gamma$. Let $\phi_1, \ldots, \phi_r: G^\eps_g \to \mathbb{R}$ and $\Phi: G^\eps_g \to \mathbb{R}$ be the functions defined by $\phi_i(h) = \lambda_1(\rho_i(h))$ and
\[\Phi(h) = \prod_{i = 1}^r \phi_i(h)^{\beta_i} - 1.\]
Thus $\Phi$ is an analytic function that vanishes on~$\Gamma$.

Let us show that $\Phi$ vanishes on~$\Delta$. Let $g_1 \cdots g_p \in \Delta$ be some word, where every $g_l$ lies in one of the subsets $M_j a_j^t$. We want to show that $\Phi(g_1 \cdots g_p) = 0$. This can be done by induction on the number of indices $l$ such that $g_l$ is not of the form $\gamma_j^n$, using the lemma below.

Thus $\Phi$ vanishes on~$\Delta$. But $\Delta$~has nonempty interior (Lemma~\ref{sec:5.1}), hence $\beta_1 = \cdots = \beta_r = 0$ as desired.

We used the following lemma.

\begin{lemma*} ($k = \mathbb{R}$)
Let $G^\eps_g$~be an open subsemigroup of~$G$ whose elements are all $\mathbb{R}$-regular, and $\Phi$~be the function defined above. Let $h$~be an element of~$G^\eps_g$ and $h_1$, $h_2$~be elements of $G^\eps_g \cup \{1\}$. Let $h = m a = a m$ be the Jordan decomposition of~$h$ with $m$~elliptic and $a$~hyperbolic. Let $M$~denote the closure of the group generated by~$m$. Suppose that
\begin{itemize}
\item for every real $t \geq 1$, $M a^t \subset G^\eps_g$;
\item for every integer $n \geq 1$, $\Phi(h_1 h^n h_2) = 0$.
\end{itemize}
Then for every real $t \geq 1$, we have $\Phi(h_1 M a^t h_2) = 0$.
\end{lemma*}

\begin{proof}
Note that $M$ is a compact group, hence a Zariski-closed subset of $G \subset \SL(n, \mathbb{R}) \subset \mathbb{R}^{n^2}$. Let us call
\[Z := \setsuch{(m', t) \in M \times [1, \infty)}{\Phi(h_1 m' a^t h_2) = 0},\]
and let $p: Z \to M$ be the first projection.

Let us first check that, for every $m_o$ in~$M$, the fiber~$p^{-1}(m_o)$ is either all of $m_o \times [1, \infty)$ or is finite. For this, let $\psi_1, \ldots, \psi_r: [1, \infty) \to \mathbb{R}$ and $\Psi: [1, \infty) \to \mathbb{R}$ be the functions defined by $\psi_i(t) = \lambda_1(\rho_i(h_1 m_o a^t h_2))$ and
\[\Psi(t) = \Phi(h_1 m_o a^t h_2) = \prod_{i = 1}^r \psi_i(t)^{\beta_i} - 1.\]
By Proposition~\ref{sec:7.2}, there exists a Hardy field that contains the germs at~$+\infty$ of the functions~$\psi_i$ and~$\Psi$. Indeed, $\psi_i(t)$ is a solution of the equation in~$X$
\[\det \left( \rho_i(h_1 m_o a^t h_2)^2 - X^2 \right) = 0\]
whose coefficients are polynomials in~$t$ and in exponential expressions~$e^{\kappa t}$ whose coefficients~$\kappa$ are real. In particular, if $\Psi$~is nonzero, it has only finitely many zeros, \ie $p^{-1}(m_o)$ is finite as desired.

Suppose by contradiction that $Z \neq M \times [1, \infty)$. We can then find a nonempty open subset $U$ of~$M$ such that for every $m_o$ in~$U$, $m_o \times [1, \infty)$ is not contained in~$Z$. By what we said above, $p^{-1}(m_o)$ is then finite. Lemma~\ref{sec:7.3} then allows us to choose~$U$ in such a way that $p^{-1}(\overline{U})$ is compact. But the sequence~$S$ of the integers~$n$ such that $m^n$~is in~$U$ is infinite. By assuption, for $n \geq 1$, we have $\Phi(h_1 m^n a^n h_2) = 0$. Hence, for $n$ in~$S$, $(m^n, n)$ is in $p^{-1}(\overline{U})$. This contradicts the compactness of $p^{-1}(\overline{U})$. Hence $Z = M \times [1, \infty)$, and the conclusion follows.
\end{proof}

\noindent {\scshape Higher School of Economics, ul Usacheva d. 6, g. Moskva, 119048, Russia} \\
\noindent {\itshape E-mail address:} \url{ilia.smilga@normalesup.org}
\end{document}